\documentclass{amsart}
\usepackage{amsfonts}
\usepackage{amsopn}
\usepackage{epsfig}
\usepackage{amsmath}
\usepackage{latexsym}
\usepackage{color}

\newtheorem{theorem}{Theorem}[section]
\newtheorem{proposition}[theorem]{Proposition}
\newtheorem{lemma}[theorem]{Lemma}
\newtheorem{definition}[theorem]{Definition}

\theoremstyle{definition}

\newtheorem{remark}[theorem]{Remark}

\newtheorem{notation}[theorem]{Notation}
\newtheorem{example}[theorem]{Example}

\newcommand{\QQ}{\mathbb{Q}}

\newcommand{\PP}{\mathbb{P}}
\newcommand{\GG}{\mathbb{G}}

\newcommand{\OO}{\mathcal{O}}

\newcommand{\ZZ}{\mathbb{Z}}

\newcommand{\Kgnb}[1]{\overline{\mathcal{M}}_{#1}}
\newcommand{\Hh}{H_{\sigma_{1,1}}}
\newcommand{\HH}{H_{\sigma_{2}}}
\newcommand{\RR}{\mathbb{R}}
\newcommand{\ddeg}{D_{deg}}
\newcommand{\dunb}{D_{unb}}
\newcommand{\Ddeg}{\mathcal{D}_{deg}}
\newcommand{\Dunb}{\mathcal{D}_{unb}}
\newcommand{\ccc}{\mathcal{C}}
\newcommand{\qqq}{\mathcal{Q}}
\newcommand{\lll}{\mathcal{L}}
\newcommand{\Hilb}{\operatorname{Hilb}}

\begin{document}

\title{Stable Base Locus Decompositions of Kontsevich Moduli Spaces}

\author{Dawei Chen}

\address{University of Illinois at
  Chicago, Mathematics Department, Chicago, IL 60607}

\email{dwchen@math.uic.edu}

\author{Izzet Coskun}

\address{University of Illinois at
  Chicago, Mathematics Department, Chicago, IL 60607}

\email{coskun@math.uic.edu}


\thanks{During the preparation of this article the second author was
  partially supported by the NSF grant DMS-0737581.}

\begin{abstract}
In this paper, we determine the stable base locus decomposition of the Kontsevich moduli spaces of degree two and three stable maps to Grassmannians. This gives new examples of the decomposition for varieties with Picard rank three. We also discuss the birational models that correspond to the chambers in the decomposition.
\end{abstract}

\maketitle
\tableofcontents

\section{Introduction}

The effective cone of a projective variety can be decomposed into chambers depending on the stable base locus of the corresponding linear series. This decomposition dictates the different birational models of the variety that arise while running the Minimal Model Program and has been studied in detail in \cite{ein:stable} and \cite{ein:restricted}. In general, especially when the dimension of the Neron-Severi space is three or more, it is very hard to compute the  decomposition. Consequently, there have been few explicit examples. In this paper,  we  completely determine the stable base locus decomposition of the Kontsevich moduli spaces $\Kgnb{0,0}(G(k,n),d)$ for $d=2$ or $3$.  Theorem \ref{de-two} contains the description of the decomposition for $\Kgnb{0,0}(G(k,n),2)$  and Theorems \ref{de-three} and \ref{de-three-line} contain the description for $\Kgnb{0,0}(G(k,n),3)$. We also describe in detail the birational models that correspond to each chamber in the decomposition.
\smallskip

Let $X$ be a complex, projective, homogeneous variety. Let $\beta$ be the homology class of a curve on $X$. Recall that the Kontsevich moduli space of $m$-pointed, genus-zero stable maps $\Kgnb{0,m}(X, \beta)$ is a compactification of the space of $m$-pointed rational curves  on $X$ with class $\beta$, parameterizing isomorphism classes of maps $(C,p_1, \dots, p_m, f)$ such that
\begin{enumerate}
\item The domain curve $C$ is a proper, connected, at-worst-nodal curve of arithmetic genus zero.
\item The marked points $p_1, \dots, p_m$ are smooth, distinct points on $C$.
\item $f_{*} [C] = \beta$ and any component of $C$ contracted by $f$ has at least three nodes or marked points.
\end{enumerate}
$\Kgnb{0,m}(X, \beta)$ is a smooth stack and the corresponding coarse moduli space is $\QQ$-factorial with finite quotient singularities. Furthermore, in $\operatorname{Pic}(\Kgnb{0,m}(G(k,n), d)) \otimes \QQ$ linear and numerical equivalence coincide. We can, therefore, construct $\QQ$-Cartier divisors by specifying codimension one conditions on $\Kgnb{0,m}(G(k,n), d)$ and calculate their class by the method of test curves.
\smallskip

The study of the stable cone decomposition of $\Kgnb{0,0}(G(k,n),d)$ has two components. On the one hand, we construct effective divisors in a given numerical equivalence class and thereby limit the stable base locus.  On the other hand, we construct moving curve classes on subvarieties of $\Kgnb{0,0}(G(k,n),d)$ that have negative intersection with a divisor class, thereby showing that the stable base locus has to contain that variety. This analysis requires a good understanding of the cones of ample and effective divisors on Kontsevich moduli spaces, which have been studied in \cite{coskun:kontamp} and \cite{coskun:konteff} when the target is $\PP^n$ and in \cite{coskun:grass} when the target is $G(k,n)$.
\smallskip

When $d=2$ or $3$, one can run the Minimal Model Program for $\Kgnb{0,0}(\PP^d,d)$ giving a complete description of the birational models $\operatorname{Proj}(\oplus_{m \geq 0} H^0 (\OO(mD))$ for every integral effective divisor. $\Kgnb{0,0}(\PP^2,2)$ is isomorphic to the space of complete conics, or equivalently, to the blow-up of $\PP^5$ along a Veronese surface. $\Kgnb{0,0}(\PP^2,2)$  admits two divisorial contractions to $\PP^5$ and $(\PP^5)^*$ obtained by projection from the space of complete conics to the spaces of conics and dual conics, respectively. The resulting models can be given functorial interpretations.  $\PP^5$ can be interpreted either as the Chow variety or the Hilbert scheme $\Hilb_{2x+1}(\PP^2)$ of conics in $\PP^2$. $(\PP^5)^*$ can be interpreted as  the moduli space of weighted stable maps $\Kgnb{0,0}(\PP^2, 1, 1)$ constructed by Anca Musta\c{t}\v{a} and Andrei Musta\c{t}\v{a} in  \cite{mustata:kontsevich}. The Mori theory  of $\Kgnb{0,0}(\PP^3, 3)$ has been studied in \cite{dawei:cubic}. $\Kgnb{0,0}(\PP^3, 3)$ admits a divisorial contraction to the moduli space of weighted stable maps $\Kgnb{0,0}(\PP^3, 2, 1)$ and a flipping contraction to the normalization of the Chow variety. The flip is the component of the Hilbert scheme $\Hilb_{3x+1}(\PP^3)$ whose general point is a twisted cubic curve. This component of the Hilbert scheme admits a further divisorial contraction to the compactification of the space of twisted cubics in $G(3,10)$ by nets of quadrics vanishing on the curve.

\smallskip

The Mori theory of $\Kgnb{0,0}(G(2,4),2)$ can be similarly described in complete detail and gives rise to some beautiful classical projective geometry. $\Kgnb{0,0}(G(2,4),2)$ admits a divisorial contraction to the space of weighted stable maps $\Kgnb{0,0}(G(2,4),1,1)$
and two intermediate contractions over $\Kgnb{0,0}(G(2,4),1,1)$ which are flops of each other (see Theorem \ref{T} for precise statements).  $\Kgnb{0,0}(G(2,4),2)$ admits a flipping contraction to the (normalization of) the Chow variety. The flip of $\Kgnb{0,0}(G(2,4),2)$ over the Chow variety is the Hilbert scheme $\Hilb_{2x+1}(G(2,4))$, which is isomorphic to the blow-up of the Grassmannian $G(3,6)$ along (both components of) the orthogonal Grassmannian $OG(3,6)$. The Hilbert scheme admits a divisorial contraction to $G(3,6)$ blowing-down the inverse image of $OG(3,6)$ and two intermediate contractions blowing-down the inverse image of only one of the components of $OG(3,6)$. The latter two spaces are flips of $\Kgnb{0,0}(G(2,4),2)$ over contractions of the (normalization of) the Chow variety (see Theorem \ref{chow} for precise statements).
\smallskip

As $d$ gets larger, both the dimension of the Neron-Severi space and the number of chambers in the decomposition of the effective cone grow.  Already for $\Kgnb{0,0}(G(3,6),3)$ there are more than 20 chambers in the decomposition. It is unreasonable to expect a complete description  of the stable base locus decomposition of $\Kgnb{0,0}(G(k,n),d)$, or even of $\Kgnb{0,0}(\PP^n,d)$, for large $d$. Note that such a description would resolve many outstanding problems in moduli theory and classical algebraic geometry,  including the F-conjecture which describes the ample cones of the Deligne-Mumord moduli spaces $\Kgnb{g,n}$ of $n$-pointed genus-$g$ stable curves and many conjectures about the syzygies and secants of rational curves and scrolls. However, the methods of this paper can be used to obtain a rough description of the stable base locus decomposition even when $d>3$. In order to keep this paper at a reasonable length, we postpone this discussion to \cite{dawei:stable}.

\smallskip

The organization of this paper is as follows: In \S \ref{prelim}, we set the notation and introduce divisor classes that will play an important role in our discussion. In \S \ref{sec-two}, we determine the stable base locus decomposition of  $\Kgnb{0,0}(G(k,n),2)$ and describe the corresponding birational models. In \S \ref{sec-three}, we carry out the same analysis for  $\Kgnb{0,0}(G(k,n),3)$ with $3 \leq k < k+3 \leq n$. The description of the stable base locus decomposition of $\Kgnb{0,0}(G(2,n),3)$ is different. We carry out the analysis in \S \ref{sec-three-line}.
\smallskip

\noindent{\bf Acknowledgements:} It is a pleasure to thank Joe Harris and Jason Starr for many enlightening conversations over the years about the birational geometry of Kontsevich moduli spaces. We also want to thank Lawrence Ein, Robert Lazarsfeld and Mihnea Popa for a number of suggestions
on this work.

\section{Preliminary definitions and background}\label{prelim}

In this section, we introduce important ample and effective divisor classes  on the Kontsevich moduli space $\Kgnb{0,0}(G(k,n), d)$.  We refer the reader to \cite{coskun:kontamp}, \cite{coskun:konteff} and \cite{coskun:grass} for detailed information about the ample and effective cones of Kontsevich moduli spaces.

\begin{notation}
Let $G(k,n)$ denote the Grassmannian of $k$-dimensional subspaces of an $n$-dimensional vector space $V$. Let $\lambda$ denote a partition with $k$ parts satisfying $n-k \geq \lambda_1 \geq \cdots \geq \lambda_k \geq 0.$ Let $\lambda^*$ denote the partition dual to $\lambda$ with parts $\lambda^*_i = n-k - \lambda_{k-i+1}$. Let $F_{\bullet}: F_1 \subset \cdots \subset F_n$ denote a flag in $V$.  The Schubert cycle $\sigma_{\lambda}$ is the Poincar\'{e} dual of the class of the Schubert variety $\Sigma_{\lambda}$ defined by
$$\Sigma_{\lambda}(F_{\bullet}) = \{  [W] \in G(k,n) \ | \ \dim(W \cap F_{n-k+i - \lambda_i}) \geq i  \}.$$
 Schubert cycles  form a $\ZZ$-basis for the cohomology of $G(k,n)$.
\end{notation}

Let $\Kgnb{0,0}(G(k,n), d)$ denote the Kontsevich moduli space of stable maps to $G(k,n)$ of Pl\"ucker degree $d$.   Let $$\pi: \Kgnb{0,1}(G(k,n), d) \rightarrow \Kgnb{0,0}(G(k,n), d)$$ be the forgetful morphism and let  $$e: \Kgnb{0,1}(G(k,n), d) \rightarrow G(k,n)$$ be the evaluation morphism. We introduce the divisor classes that will be crucial for our discussion.

\noindent $\bullet$ Let  $\Hh= \pi_* e^* (\sigma_{1,1})$ and $\HH= \pi_* e^* (\sigma_{2})$. Geometrically, $\Hh$ (resp., $\HH$) is the class of the divisor of maps $f$ in $\Kgnb{0,0}(G(k,n), d)$ whose image intersects a fixed Schubert cycle $\Sigma_{1,1}$ (resp., $\Sigma_2$).
\smallskip

\noindent $\bullet$ Let $T_{\sigma_1}$ denote the class of the divisor of maps that are tangent to a fixed hyperplane section of $G(k,n)$.

\noindent $\bullet$ Let $\ddeg$ denote the class of the divisor $\Ddeg$ of maps in $\Kgnb{0,0}(G(k,k+d), d)$ whose image is contained in a sub-Grassmannian $G(k,k+d-1)$ embedded in $G(k, k+d)$ by an inclusion of the ambient vector spaces. More generally, for $n \geq k+d$, let $\ddeg$ denote the class of the divisor  of maps  $f$ in $\Kgnb{0,0}(G(k,n),d)$ such that the projection of the span of the linear spaces parameterized by the image of $f$ from a fixed linear space of dimension $n-k-d$ has dimension less than $k+d$.
\smallskip

\noindent $\bullet$ If $k$ divides $d$, then let $\dunb$ be the closure $\Dunb$ of the locus of maps  $f$ with irreducible domains for which the pull-back of the tautological bundle $f^*(S)$ has unbalanced splitting (i.e., $f^*(S) \not= \oplus_{i=1}^k \OO_{\PP^1}(-d/k)$). If $k$ does not divide $d$, let $d=k q + r$, where $r$ is the smallest non-negative integer that satisfies the equality. The subbundle of the pull-back of the tautological bundle of rank $k-r$ and degree $-q(k-r)$ induces a rational map $$\phi: \Kgnb{0,0}(G(k,k+d), d) \dashrightarrow \Kgnb{0,0}(F(k-r, k ; k+d), q(k-r), d).$$ The natural projection $\pi_{k-r}: F(k-r, k ; k+d) \rightarrow G(k-r, k+d)$ from the two-step flag variety to the Grassmannian induces a morphism $$\psi: \Kgnb{0,0}(F(k-r, k , k+d), q(k-r), d) \rightarrow \Kgnb{0,0}(G(k-r,  k+d), q(k-r)).$$ The maps whose linear spans intersect a linear space of codimension  $(q+1)(k-r)$ is a divisor $D$ in $\Kgnb{0,0}(G(k-r,  k+d), q(k-r))$.  Let $\dunb= \phi^* \psi^* ([D])$.
\smallskip

We summarize the basic facts about the Picard group, and the cones of NEF and effective divisors in the following theorem.

\begin{theorem}\label{summary} Let $\Kgnb{0,0}(G(k,n), d)$ denote the Kontsevich moduli space of stable maps to $G(k,n)$ of Pl\"ucker degree $d$.  Then:
\end{theorem}

\noindent $(1)$ [Theorem 1, \cite{dragos:taut}] {\it The Picard group $\mbox{Pic}(\Kgnb{0,0}(G(k,n),d))\otimes \QQ$ is generated by the divisor classes $\Hh$, $\HH$, and the classes of the boundary divisors $\Delta_{k,d-k}$, $1 \leq k \leq \lfloor d/2 \rfloor$. }
\smallskip

\noindent $(2)$ [Theorem 1.1, \cite{coskun:grass}] {\it There is an explicit, injective linear map $$v: \mbox{Pic}(\overline{M}_{0,d}/ \mathfrak{S}_d)\otimes \QQ \rightarrow \mbox{Pic}(\Kgnb{0,0}(G(k,n),d))\otimes \QQ$$ that maps base-point-free divisors and NEF divisors to base-point-free divisors and NEF divisors, respectively. A divisor class $D$ in $\Kgnb{0,0}(G(k,n), d)$ is NEF if and only if $D$ can be expressed as a non-negative linear combination of $\Hh, \HH, T_{\sigma_1}$ and $v(D')$, where $D'$ is a NEF divisor in $\overline{M}_{0,d}/ \mathfrak{S}_d$.}
\smallskip

\noindent $(3)$ [Theorem 1.2, \cite{coskun:grass}] {\it  A divisor class $D$ in $\Kgnb{0,0}(G(k, k+d), d)$ is effective if and only if it can be expressed as a non-negative linear combination of} $\ddeg$, $\dunb$ { \it and the boundary divisors $\Delta_{k,d-k}$, $1 \leq k \leq \lfloor d/2 \rfloor$.}
\smallskip

\begin{remark}
In Part (2) of Theorem \ref{summary},  $\overline{M}_{0,d}$ denotes the Deligne-Mumford moduli space of $d$-pointed, genus-zero stable curves. The symmetric group $\mathfrak{S}_d$ on $d$-letters acts on the labeling of the marked points.

If we identify the Neron-Severi space of $\Kgnb{0,0}(G(k,n),d)$ with the vector space spanned by the divisor classes $\Hh$, $\HH$, and the classes of the boundary divisors $\Delta_{k,d-k}$, $1 \leq k \leq \lfloor d/2 \rfloor$, then the effective cone of $\Kgnb{0,0}(G(k,n),d)$  is contained in the effective cone of $\Kgnb{0,0}(G(k,n+1),d)$, with equality if $n \geq k+d$. Hence, Part (3) of Theorem \ref{summary} determines the effective cone of $\Kgnb{0,0}(G(k,n),d)$ for every $n \geq k+d$.
\end{remark}
\smallskip

\begin{remark}
The canonical class of $\Kgnb{0,0}(G(k,n),d)$ can be easily derived from Proposition 7.2 of \cite{jason:canonical}: $$K = \left( \frac{n}{2} - k -1 - \frac{n}{2d} \right) \Hh + \left( k - \frac{n}{2} - 1 - \frac{n}{2d} \right) \HH + \sum_{i=1}^{\lfloor d/2 \rfloor} \left( \frac{n i (d-i)}{2d} -2 \right) \Delta_i.$$ For most of the cases we consider below, $-K$ will not lie in the ample cone.
\end{remark}

\section{Degree two maps to Grassmannians}\label{sec-two}

In this section, we let  $2 \leq k < k+2 \leq n$ and discuss the stable base locus decomposition of $\Kgnb{0,0}(G(k,n),2)$. The divisor classes introduced in Section \ref{prelim} have the following expressions (see \cite{coskun:grass} \S 4 and \S 5) in terms of the basis $\Hh, \HH$ and the boundary divisor $\Delta= \Delta_{1,1}$.

\begin{eqnarray*}
T &=& \frac{1}{2}\left(\Hh + \HH + \Delta \right) \\
\ddeg &=& \frac{1}{4}(-\Hh + 3 \HH - \Delta) \\
\dunb &=& \frac{1}{4}(3\Hh - \HH - \Delta)
\end{eqnarray*}

Most questions about the divisor theory of $\Kgnb{0,0}(G(k,n),2)$ can be reduced to studying the divisor theory of $\Kgnb{0,0}(G(2,4),2)$. Let $W$ be a four-dimensional subspace of $V$. Let $U$ be a $(k-2)$-dimensional subspace of $V$ such that $U \cap W = 0$. Given a two-dimensional subspace $\Lambda$ of $W$, the span of $\Lambda$ and $U$ is a $k$-dimensional subspace of $V$.  Hence, there is an inclusion $i: G(2,4) \rightarrow G(k,n)$, which induces a morphism $$\phi: \Kgnb{0,0}(G(2,4),2) \rightarrow \Kgnb{0,0}(G(k,n),2).$$ It is easy to see that $$\phi^*(\Hh) = \Hh, \ \phi^*(\HH) = \HH, \ \phi^*(\Delta) = \Delta.$$ Under this correspondence, the stable base-locus-decomposition of $\Kgnb{0,0}(G(k,n),2)$ and  $\Kgnb{0,0}(G(2,4),2)$ coincide. This correspondence guides many of the constructions in this section.

\begin{remark}
The geometry of $\Kgnb{0,0}(G(2,4),2)$ is closely related to the geometry of quadric surfaces in $\PP^{3}$. The lines parameterized by a point in $\Kgnb{0,0}(G(2,4),2)$ sweep out a degree two surface in $\PP^{3}$. The maps parameterized by $\Ddeg$ correspond to those that span a plane two-to-one. The maps parameterized by $\Dunb$ correspond to those that sweep out a quadric cone.
\end{remark}

In order to understand the stable base locus decomposition of $\Kgnb{0,0}(G(k,n),2)$, we need to introduce one more divisor class. Set $N = {n \choose k}$. Let $p: \Kgnb{0,0}(G(k,n),2) \dashrightarrow G(3,N)$ denote the rational map, defined away from the locus of double covers of a line in $G(k,n)$, sending a stable map to the $\PP^2$ spanned by its image in the Pl\"{u}cker embedding of $G(k,n)$. This map gives rise to a well-defined map $p^*$ on Picard groups. Let $P = p^*(\OO_{G(3,N)}(1))$. Geometrically, $P$ is the class of the closure of the locus  of maps $f$ such that the linear span of the image of $f$ (viewed in the Pl\"{u}cker embedding of $G(k,n)$) intersects a fixed codimension three linear space  in $\PP^{N-1}$.

\begin{lemma}\label{p-two}
The divisor class $P$ is equal to
$$P = \frac{1}{4}(3\Hh + 3\HH - \Delta).$$
\end{lemma}

\begin{proof}
The divisor class $P$ can be computed by intersecting with test families. Let $\lambda = (1,1)^*$ and $\mu= (2)^*$ be the partitions dual to $(1,1)$ and  $(2)$, respectively.  In the Pl\"ucker embedding of $G(k,n)$, both $\Sigma_{\lambda}$ and $\Sigma_{\mu}$ are linear spaces of dimension two. Let $C_1$ and $C_2$, respectively,  be the curves in $\Kgnb{0,0}(G(k,n),2)$ induced by  a general pencil of conics in a fixed $\Sigma_{\lambda}$, respectively, $\Sigma_{\mu}$. Let $\tilde{C}_3$ be the curve in $\Kgnb{0,0}(G(2,4),2)$ induced by a general pencil of conics in a general codimension two linear section of $G(2,4)$ in its Pl\"{u}cker embedding. Let $C_3 = \phi(\tilde{C}_3)$. The following intersection numbers are easy to compute.
$$C_1 \cdot \Hh = 1, \  \  C_1 \cdot \HH = 0, \  \   C_1 \cdot \Delta = 3, \ \ C_1 \cdot P = 0$$
$$C_2 \cdot \Hh = 0, \  \  C_2 \cdot \HH = 1, \  \   C_2 \cdot \Delta = 3, \ \ C_2 \cdot P = 0$$
$$C_3 \cdot \Hh = 1,  \  \ C_3 \cdot \HH = 1, \ \ C_3 \cdot \Delta = 2, \ \ C_3 \cdot P = 1$$
Let $a\sigma_{\lambda} + b \sigma_{\mu}$ be the cohomology class of the surface swept out by the images of the maps parameterized by a curve $C$ in $\Kgnb{0,0}(G(k,n),2)$. Then the intersection number of  $C$  with $\Hh$ (resp., $\HH$) is equal to $a$ (resp., $b$). Since $C_1, C_2$ and $C_3$  sweep out  surfaces with cohomology class $\sigma_{\lambda}, \sigma_{\mu}$ and  $\sigma_{\lambda} + \sigma_{\mu}$, respectively, the intersection numbers of these curves with $\Hh$ and $\HH$ are as claimed. A general pencil of conics in the plane has three reducible elements. A general pencil of conics in a quadric surface has two reducible elements. Since the total space of the surfaces are smooth at the nodes, the intersections with the boundary divisor are transverse. Therefore, the intersection numbers of the curves $C_i$ with $\Delta$ are as claimed. Finally, the intersection numbers of the curves $C_i$ with $P$ are clear.
The class $P$ is determined by these intersection numbers.
\end{proof}
\smallskip

\begin{notation}
Given two divisor classes  $D_1, D_2$, let $c(D_1 D_2)$ (respectively, $c(\overline{D}_1 \overline{D}_2)$) denote the open (resp., closed) cone in the Neron-Severi space spanned by positive (resp., non-negative) linear combinations of $D_1$ and $D_2$.  Let  $c(D_1\overline{D}_2)$ denote the cone spanned by linear combinations $$c(D_1\overline{D}_2)= \{aD_1 + bD_2 | a \geq 0, b>0\}.$$
The domain in $\RR^3$ bounded by the divisor classes $D_1, D_2, \dots, D_r$ is the open domain bounded by $c(\overline{D}_1\overline{D}_2), c(\overline{D}_2 \overline{D}_3), \dots, c(\overline{D}_r \overline{D}_1)$.
\end{notation}

\begin{definition}
Let $\qqq[\lambda]$ denote the closure of the locus of maps $f$ in $\Kgnb{0,0}(G(k,n),2)$ with irreducible domain such that the map $f$  factors through the inclusion of some Schubert variety $\Sigma_{\lambda}$ in $G(k,n)$.
\end{definition}

\begin{example}
For example, $\qqq[(1)^*]$ denotes the locus of maps two-to-one onto a line in the Pl\"{u}cker embedding of $G(k,n)$. The union of $\qqq[(1,1)^*]$ and $\qqq[(2)^*]$ in $\Kgnb{0,0}(G(k,n),2)$ is the  locus of maps $f$ such that the span of $f$ is contained in $G(k,n)$. The linear spaces parameterized by a general map in $\qqq[(1,1)^*]$ sweep out a $\PP^k$ two-to-one. The linear spaces parameterized by a general map in $\qqq[(2)^*]$ sweep out a $k$-dimensional cone over a conic curve.
\end{example}

Theorem \ref{de-two} and Figure \ref{fig:conedecomp} describe the eight chambers in the stable base locus decomposition of $\Kgnb{0,0}(G(k,n),2)$.  In the figure, we draw a cross-section of the three-dimensional cone and mark each chamber with the corresponding number that describes the chamber in the theorem.

\begin{figure}[ht!] 
   \centering
   \includegraphics[width=4in]{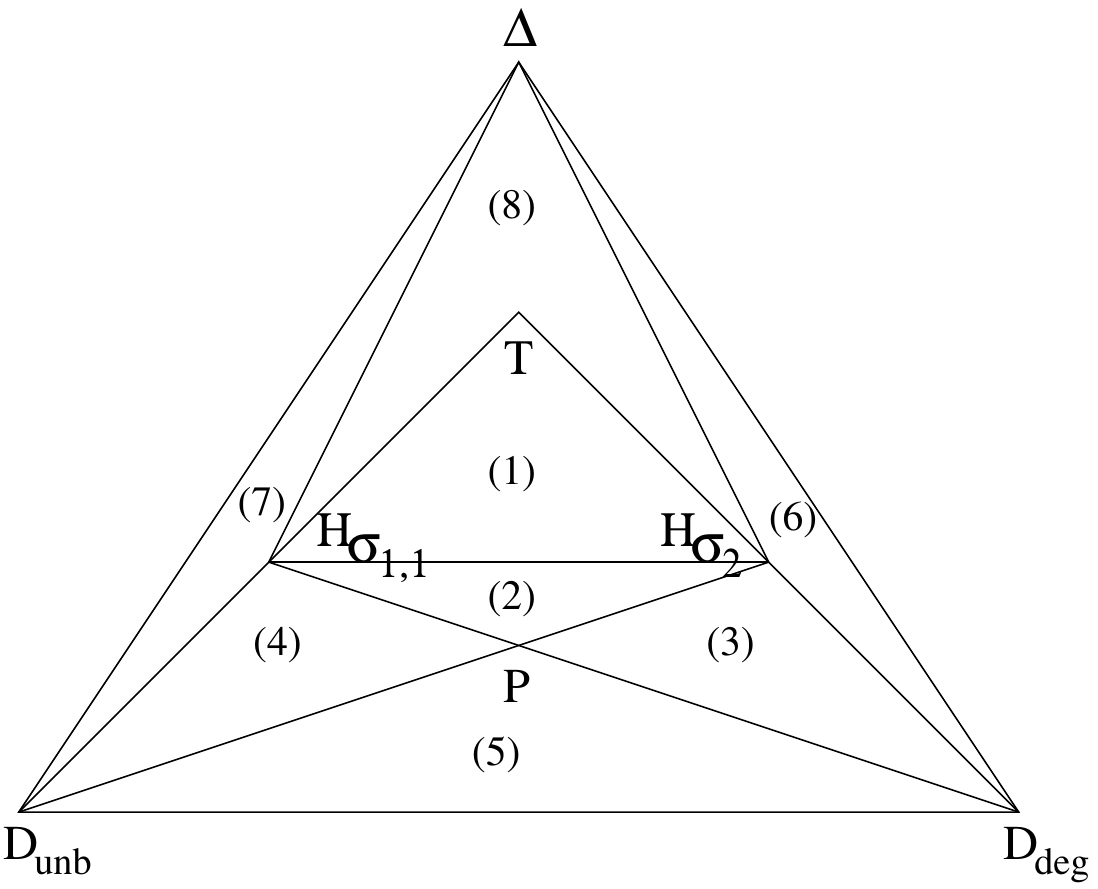}
   \caption{The stable base locus decomposition of $\Kgnb{0,0}(G(k,n),2)$.}
   \label{fig:conedecomp}
\end{figure}

\begin{theorem}\label{de-two}
The stable base locus decomposition partitions the effective cone of $\Kgnb{0,0}(G(k,n),2)$ into the following chambers:

\begin{enumerate}
\item In the closed cone spanned by non-negative linear combinations of $\Hh, \HH$ and $T$, the stable base locus is empty.
\smallskip

\item In the domain bounded by  $\Hh, \HH$ and $P$ union  $c(\Hh \overline{P}) \cup c(\HH  \overline{P})$, the stable base locus consists of the locus $\qqq[(1)^*]$ of maps two-to-one onto a line in $G(k,n)$.
\smallskip

\item In the domain  bounded by  $\HH,\ddeg$ and $P$ union  $c(\HH \overline{D}_{deg}) \cup c(P \overline{D}_{deg})$, the stable base locus consists of the locus $\qqq[(1,1)^*]$.
\smallskip

\item In the domain bounded by $\Hh, D_{unb}$ and $P$  union  $c(\Hh \overline{D}_{unb}) \cup c(P \overline{D}_{unb})$,  the stable base locus consists of the locus $\qqq[(2)^*]$.
\smallskip

\item In the domain bounded by  $P, \ddeg$ and $\dunb$ union   $c(\ddeg \dunb)$,  the stable base locus consists of the union  $\qqq[(1,1)^*] \cup \qqq[(2)^*]$.
\smallskip

\item In the domain bounded by $\HH, \ddeg$ and $\Delta$ union  $c(\ddeg \Delta)$, the stable base locus consists of the union of the boundary divisor and  $\qqq[(1,1)^*]$.
\smallskip

\item In the domain bounded by  $\Hh, \dunb$ and $\Delta$ union $c(\dunb \Delta)$, the stable base locus consists of the union of the boundary divisor and $\qqq[(2)^*]$.
\smallskip

\item Finally, in the domain  bounded by $ \Hh, T, \HH$ and $\Delta$  union $c(\HH \overline{\Delta}) \cup c(\Hh \overline{\Delta})$ the stable base locus consists of  the boundary divisor.
\end{enumerate}

\end{theorem}
\smallskip

\begin{proof}
Since the effective cone of $\Kgnb{0,0}(G(k,n), 2)$ is generated by non-negative linear combinations of $\ddeg$, $\dunb$ and  $\Delta$, the stable base locus of any divisor has to be contained in the union of the stable base loci of $\ddeg$, $\dunb$ and the boundary divisor.  We first check that the loci described in the theorem are in the stable base locus of the claimed divisors.  To show that a variety $X$ is in the base locus of a linear system $|D|$, it suffices to cover $X$ by curves $C$ that have negative intersection with $D$.
\smallskip

Express a general divisor $D= a \Hh + b \HH + c\Delta$.  Recall from the proof of Lemma \ref{p-two} that $C_1$ and $C_2$ are the curves induced by pencils of conics in $\Sigma_{\lambda} $ and $\Sigma_{\mu}$, respectively, where $\lambda= (1,1)^*$ and $\mu= (2)^*$. The intersection numbers of $C_1$ and $C_2$ with $D$ are
$$C_1 \cdot D = a + 3c, \  \  \  \  C_2 \cdot D = b + 3c.$$ Since curves in the class $C_1$ (resp., $C_2$) cover $\qqq[(1,1)^*]$ (resp., $\qqq[(2)^*]$), we conclude that $\qqq[(1,1)^*]$ (resp., $\qqq[(2)^*]$) is in the base locus of the linear system $|D|$ if $a + 3c <0$ (resp., $b + 3c<0$). In other words, $\qqq[(1,1)^*]$ is in the base locus of the divisors contained in the interior of the cone generated by  $\ddeg, \dunb$ and $\ddeg + \Delta/3$ and in $c(\dunb \overline{D}_{deg})$. Similarly, $\qqq[(2)^*]$ is in the base locus of a divisor contained in the interior of the cone generated by $\ddeg, \dunb$ and $\dunb + \Delta/3$ and in $c(\ddeg \overline{D}_{unb})$.
\smallskip

Let $C_4$ and $C_5$ be the curves induced in $\Kgnb{0,0}(G(k,n), 2)$ by the one parameter family of conics tangent to four general lines in a fixed $\Sigma_{\lambda}$ and $\Sigma_{\mu}$, respectively.  It is straightforward to see that $$C_4 \cdot D= 2a, \ \ \ \  C_5 \cdot D = 2b.$$ Curves of type $C_4$ and $C_5$ cover $\qqq[(1,1)^*]$ and $\qqq[(2)^*]$, respectively. Consequently, if $a < 0$ (resp., $b<0$) $\qqq[(1,1)^*]$ (resp., $\qqq[(2)^*]$) is in the base locus of $|D|$. We conclude that $\qqq[(1,1)^*]$ is in the base locus of any divisor contained in the region bounded by $\ddeg, \Delta, \HH$ and $\dunb$ and in $c(\Delta \overline{D}_{deg}) \cup c(\dunb \overline{D}_{deg})$.  Similarly, $\qqq[(2)^*]$ is in the base locus of any divisor contained in the region bounded by $\dunb, \Delta, \Hh$ and $\ddeg$ and in $c(\Delta \overline{D}_{unb}) \cup c(\ddeg \overline{D}_{unb})$.
\smallskip

Next let $C_6$ and $C_7$ be the curves induced by attaching a line at the base point of  a pencil of lines in $\Sigma_{\lambda}$ and $\Sigma_\mu$, respectively. These curves have the following intersection numbers with $D$:
$$C_6 \cdot D = a - c, \ \ \ \ C_7\cdot D = b - c.$$ Since deformations of the curves in the same class as $C_6$ and $C_7$ cover the boundary divisor, we conclude that the  boundary divisor is in  the base locus of $|D|$ if $a -c < 0$ or if $b-c<0$.  Hence, the boundary divisor is in the base locus of the divisors contained in the region bounded by $\dunb,  T, \ddeg$ and  $\Delta$ and in $c(D_{unb}\overline{\Delta}) \cup c(D_{deg}\overline{\Delta})$.
\smallskip

Finally, consider the one-parameter family $C_8$ of two-to-one covers of a line $l$ in $G(k,n)$ branched along a fixed point $p \in l$ and a varying point $q \in l$. Then $$C_8 \cdot D = c.$$ Curves in the class $C_8$ cover the locus of double covers of a line. Hence, if $c<0$, then the locus of double covers of a line have to be contained in the base locus. Note that since the locus of double covers of a line is contained in  both $\qqq[(1,1)^*]$ and $\qqq[(2)^*]$, any divisor containing the latter in the base locus also contains the locus of double covers. Hence, the locus of double covers is contained in the base locus of every effective divisor contained in the complement of the closed cone generated by $\Hh, \HH$ and $\Delta$. In particular, this locus is contained in the base locus of divisors contained in the region bounded by $\Hh, \HH$ and $P$ and in $c(\Hh \overline{P}) \cup c(\HH \overline{P})$.
\smallskip

We have verified that the loci described in the theorem are in the base locus of the corresponding divisors.  We will next show that the divisors listed in the theorem contain only the listed loci in their stable base locus.  The divisors $\Hh$, $\HH$ and $T$ are base-point-free (\cite{coskun:grass} \S 5). Hence, for divisors contained in the closed cone generated by $\Hh, \HH$ and $T$ the base locus is empty. Next, note that the base locus of the linear system $|P|$ is exactly the locus of double covers of a line. The rational map $p$ in the definition of $P$ is a morphism in the complement of the locus of double covers of a line. If the image of a map $f$ is a degree two curve in $G(k,n)$, then in the Pl\"{u}cker embedding of $G(k,n)$ the image spans a unique plane. In $\PP^{N-1}$, we can always find a codimension three linear space $\Gamma$ not intersecting $\Lambda$. Hence, $f$ is not in the indeterminacy locus of the map to $G(3,N)$ and there is a section of $\OO_{G(3,N)}(1)$ not containing the image of $f$. It follows that $f$ is not in the base locus of $|P|$. By the previous paragraph, the locus of degree two maps onto a line is in the base locus of $P$. We conclude that in the region bounded by  $P, \Hh$ and $\HH$ and in $c(\HH \overline{P}) \cup c(\Hh \overline{P})$ the stable base locus consists of the  locus of double covers of a line.  For a divisor contained in the region bounded by  $\dunb, P$ and $\Hh$ and in $c(P \overline{D}_{unb}) \cup c(\Hh  \overline{D}_{unb})$, the stable base locus must be contained in the stable base locus of $\dunb$ since every divisor in this region is a non-negative linear combination of $\dunb$ and base-point-free divisors. Similarly, for a divisor contained in the region bounded by $\ddeg, P $ and $\HH$ and in $c(P \overline{D}_{deg}) \cup c(\HH  \overline{D}_{deg})$, the base locus must be contained in the stable base locus $\ddeg$. In the region bounded by  $\ddeg, \dunb$ and $P$ and in $c(\dunb \ddeg)$, the base locus must be contained in the union of the stable base loci of $\ddeg$ and $\dunb$. The (stable) base locus of $\ddeg$ is $\qqq[(1,1)^*]$ and the (stable) base locus of $\dunb$ is $\qqq[(2)^*]$. The linear spaces parameterized by a degree two map to $G(k,n)$ span a linear space of dimension at most $k+2$. As long as they span a linear space of dimension $k+2$, then the projection from a general linear space of codimension $k+2$ still spans a linear space of dimension $k+2$, hence the corresponding map is not in the base locus of $\ddeg$. Similarly, as long as the intersection of all the linear spaces parameterized by the degree two map does not contain a $k-1$ dimensional linear space, then the map is not contained in the base locus of $\dunb$. Hence, the claims in parts (3), (4) and (5) of the theorem follow.  Similarly, in the region bounded by $\dunb, \Delta$ and $\Hh$ and in $c(\dunb \Delta)$, the base locus must be contained in the union of $\qqq[(2)^*]$ and the boundary divisor. In the region bounded by $\ddeg$, $\Delta$ and $\HH$ and in $c(\ddeg \Delta)$, the base locus must be contained in the union of $\qqq[(1,1)^*]$ and the boundary divisor. We conclude the equality in these two cases as well. Finally, in the region bounded by $\Delta, \Hh$ and $\HH$ the base locus has to be contained in the boundary divisor. Hence in the complement of the closed cone spanned by $\Hh, T$ and $\HH$ the base locus must equal the boundary divisor by the calculations above. This completes the proof of the theorem.
\end{proof}
\smallskip

Next,  we describe the birational models of $\Kgnb{0,0}(G(k,n),2)$ that correspond to the chambers in the decomposition.  For a big  rational  divisor class $D$, let $\phi_{D}$ denote the birational map $$ \phi_D: \Kgnb{0,0}(G(k,n), 2) \dashrightarrow \operatorname{Proj} (\oplus_{m\geq 0} (H^0(\OO(\lfloor mD \rfloor))).$$

\begin{proposition}
The Kontsevich moduli space $\Kgnb{0,0}(G(k,n),2)$ admits the following morphisms:
\begin{enumerate}
\item $\phi_{t \Hh + (1-t) \HH}$,  for $0 < t < 1$,  is a morphism from $\Kgnb{0,0}(G(k,n),2)$ to the normalization of the Chow variety, which is an isomorphism in the complement of $\qqq[(1)^*]$, the locus of double covers of a line in $G(k,n)$, and contracts  $\qqq[(1)^*]$ so that the locus of double covers with the same image line maps to a point. \smallskip

\item $\phi_{\Hh}$ and $\phi_{\HH}$ give two morphisms from $\Kgnb{0,0}(G(k,n),2)$ to two contractions of the normalization of the Chow variety, where $\phi_{\Hh}$ (resp., $\phi_{\HH}$), in addition to the double covers of a line, also contract the boundary divisor and $\qqq[(2)^*]$ (resp., $\qqq[(1,1)^*]$).  Any two maps $f, f'$ in the boundary for which the image is contained in the union of the same Schubert varieties  $\Sigma_{(2)^*}$ (resp., $\Sigma_{(1,1)^*}$) map to the same point under $\phi_{\Hh}$ (resp., $\phi_{\HH}$). Similarly, the stable maps in $\qqq[(2)^*]$ (resp., $\qqq[(1,1)^*]$) with image contained in a fixed Schubert variety $\Sigma_{(2)^*}$ (resp., $\Sigma_{(1,1)^*}$) map to the same point under $\phi_{\Hh}$ (resp., $\phi_{\HH}$).
\smallskip

\item If $D$ is in the domain bounded by $\Hh, \HH$ and $T$, then $D$ is ample and gives rise to an embedding of $\Kgnb{0,0}(G(k,n),2)$.
\end{enumerate}
\end{proposition}

\begin{proof}
By \cite{coskun:grass}, the NEF cone of $\Kgnb{0,0}(G(k,n),2)$, which coincides with the base-point-free cone,  is the closed cone spanned by $\Hh, \HH$ and $T$.  We, therefore, obtain morphisms for sufficiently high and divisible multiples of each of the rational divisors in this cone. The last part of the proposition follows by Kleiman's Theorem which asserts that the interior of the NEF cone is the ample cone. The curves in the class $C_8$ have intersection number zero with any divisor of the form $t\Hh + (1-t) \HH$. Since these curves cover the locus of double covers of a fixed line, we conclude that the maps obtained from these divisors contract the locus of double covers of a fixed line to a point.  The class $H$ of the divisor of maps whose image intersects a codimension two linear space  in  projective space gives rise to the Hilbert-Chow morphism on $\Kgnb{0,0}(\PP^{N-1},2)$. This morphism has image the normalization of the Chow variety and is an isomorphism away from the locus maps two-to-one onto their image. The Pl\"{u}cker embedding of $G(k,n)$ induces an inclusion of $\Kgnb{0,0}(G(k,n),2)$ in $\Kgnb{0,0}(\PP^{N-1},2)$. The pull-back of $H$ under this inclusion is $\Hh + \HH$. By symmetry, there is no loss of generality in assuming that $0<t \leq 1/2$. We can write $$t \Hh + (1-t) \HH = t ( \Hh + \HH)     + (1-2t) \HH.$$ Since $\HH$ is base-point-free, the first part of  the proposition follows. The cases of $\phi_{\Hh}$ and $\phi_{\HH}$ are almost identical, so we concentrate on $\phi_{\Hh}$. $\Hh$ has intersection number zero with the curve classes $C_5, C_7$ and $C_8$. Curves in the class $C_5$ cover the locus $\qqq[(2)^*]$. Curves in the class $C_7$ cover the boundary divisor and curves in the class $C_8$ cover $\qqq[(1)^*]$. We conclude that these loci are contracted by $\phi_{\Hh}$. Part (2) of the proposition follows from these considerations. We observe that the locus of degree two curves whose span does not lie in $G(k,n)$ admit three distinct Chow compactifications depending on whether one uses the codimension two class $\sigma_{1,1}$, $\sigma_2$ or $a \sigma_{1,1} + b \sigma_2$ with $a, b>0$. The three models are the normalization of these Chow compactifications.
\end{proof}
\smallskip

\begin{theorem}\label{T}
\begin{enumerate}
\item The birational model corresponding to the divisor $T$ is the space of weighted stable maps $\Kgnb{0,0}(G(k,n), 1, 1)$. $\phi_{T}$ is an isomorphism  away from the boundary divisor and contracts the locus of maps with reducible domain $f: C_1 \cup C_2 \rightarrow G(k,n)$  that have $f(C_1 \cap C_2) = p$ for some fixed $p \in G(k,n)$  to a point.
\smallskip

\item For $D \in c(\Hh T)$ or $D \in c(\HH T)$ the morphism $\phi_D$ is an isomorphism away from the boundary divisor. On the boundary divisor, for  $D \in c(\Hh T)$ (resp., in $c(\HH T)$) the morphism contracts the locus of line pairs that are contained in the same pair of intersecting  linear spaces with class $\Sigma_{n-k-1, \dots, n-k-1}$ (resp., $\Sigma_{n-k, \dots, n-k}$) to a point. These morphisms are flops of each other over $\phi_T$.

\end{enumerate}
\end{theorem}
\smallskip

\begin{proof}
The curves in the class $C_6$ (respectively, $C_7$) have intersection number zero with a divisor class $D$ in $c(\overline{\Hh} \overline{T})$ (respectively, with $D$ in $c(\overline{\HH} \overline{T})$). The descriptions of the morphisms follow easily noting that $\phi_D$ contracts these curves. Note that the further contractions to the image of $\phi_T$ are small contractions. It is easy to check that they are flopping contractions.
\end{proof}
\smallskip

For the next lemma and theorem, we assume that the target is $G(2,4)$.
\smallskip

\begin{lemma}
Let $OG_{\sigma_{1,1}}(G(2,4))$ and $OG_{\sigma_2}(G(2,4))$ denote the two connected components of the orthogonal Grassmannian $OG(3,6)$ parametrizing projective planes contained  in the Pl\"{u}cker embedding of $G(2,4)$. Then the Hilbert scheme $Hilb_{2x-1}(G(2,4))$ corresponding to the Hilbert polynomial $2x-1$ is isomorphic to the blow-up of $G(3,6)$ along $OG(3,6)$.  The blow-down morphism $$\pi: Hilb_{2x-1}(G(2,4)) \rightarrow G(3,6)$$ factors through $$\pi_{1,1} : Hilb_{2x-1}(G(2,4)) \rightarrow Bl_{OG_{\sigma_{1,1}}}G(3,6)$$ and $$\pi_{2} : Hilb_{2x-1}(G(2,4)) \rightarrow Bl_{OG_{\sigma_{2}}}G(3,6).$$
\end{lemma}
\smallskip

\begin{proof}
Consider the universal family  $I \subset G(3,6) \times \PP^5$ over the Grassmannian admitting two natural projections $\phi_1$ and $\phi_2$ to $G(3,6)$ and $\PP^5$, respectively. The bundle $\phi_{1*} \phi_2^* \OO_{\PP^5}(2)$ is naturally identified with $Sym^2 S^*$. Since $OG(3,6)$ is defined by the vanishing of a general section of $\phi_{1*} \phi_2^* \OO_{\PP^5}(2)$ , we can identify the normal bundle of $OG(3,6)$ at a point $\Lambda$ of $OG(3,6)$ with $Sym^2 S^*|_{\Lambda}$.
$Hilb_{2x-1}(\PP^5)$ is naturally identified with $\PP (Sym^2 (S^*)) \rightarrow G(3,6).$ Then $Hilb_{2x-1}(G(2,4))$ is then given by $$\{ ([C], [\Lambda]) \  | \  [\Lambda] \in G(3,6), C \subset \Lambda \cap G(2,4),  [C] \in Hilb_{2x-1}(\Lambda) \}.$$ The projection to $G(3,6)$ is clearly an isomorphism away from $OG(3,6)$. Over $OG(3,6)$ the fiber of the Hilbert scheme is identified with the projectivization of $Sym^2 S^*$. It follows that $Hilb_{2x-1}(G(2,4))$ is isomorphic to the blow-up of $G(3,6)$ along $OG(3,6)$. Since $OG(3,6)$ has two connected components, this leads to two exceptional divisors that can be blown-down independently. The Lemma follows from these considerations.
\end{proof}
\smallskip

\begin{theorem}\label{chow}
The rational maps corresponding to the divisors $D$ in the cone generated by $\Hh, \HH$ and $P$ are as follows.
\begin{enumerate}
\item Let $0 < t <1$.  The Hilbert scheme $Hilb_{2x-1}(G(2,4))$ is the flip of $\Kgnb{0,0}(G(2,4),2)$ over the Chow variety $Chow_{t \Hh + (1-t)\HH}$.  $\phi_D: \Kgnb{0,0}(G(2,4),2) \dashrightarrow Hilb_{2x-1}(G(2,4))$ for $D $ in the domain bounded by $\Hh, \HH$ and $P$.
\smallskip

\item For $D \in c(\Hh P)$, $\phi_D : \Kgnb{0,0}(G(2,4),2) \dashrightarrow Bl_{OG_{\sigma_{1,1}}}G(3,6)$.
\smallskip

\item For $D \in c( \HH P )$, $\phi_D : \Kgnb{0,0}(G(2,4),2) \dashrightarrow Bl_{OG_{\sigma_{2}}}G(3,6)$.
\smallskip

\item $\phi_P : \Kgnb{0,0}(G(2,4),2) \dashrightarrow G(3,6)$.

\end{enumerate}
\end{theorem}
\smallskip

\begin{proof}
Consider the incidence correspondence consisting of triples $(C, C^*, \Lambda)$, where $\Lambda$ is a plane in $\PP^5$, $C$ is a connected, arithmetic genus zero, degree two curve in $G(2,4) \cap \Lambda$ and $C^*$ is a dual conic of $C$ in $\Lambda$. This incidence correspondence admits a map both to $\Kgnb{0,0}(G(2,4),2)$ and to $Hilb_{2x-1}(G(2,4))$ by projection to the first two and by projection to the first and third factors, respectively. The projection to the first factor gives a morphism to the Chow variety. Note that this projection is an isomorphism away from the locus where $C$ is supported on a line. The morphism to the Chow variety is a small contraction in the case of both the Hilbert scheme and the Kontsevich moduli space. The fiber over a point corresponding to a double line in the morphism from the Hilbert scheme to the Chow variety is isomorphic to $\PP^1$ corresponding to the choice of plane $\Lambda$ everywhere tangent to the Pl\"{u}cker embedding of $G(2,4)$ in $\PP^5$. The fiber over a point corresponding to a double line in the morphism from $\Kgnb{0,0}(G(2,4),2)$ to the Chow variety is isomorphic to $\PP^2= Sym^2 (\PP^1)$ corresponding to double covers of $\PP^1$.  Note in both the Hilbert scheme and the Kontsevich moduli space the morphisms to the Chow variety are small contractions. The locus of double lines in the Hilbert scheme (respectively, in $\Kgnb{0,0}(G(2,4),2)$) has codimension 3 (respectively, 2).
Finally, note that for $D$ in the domain bounded by $\Hh, \HH$ and $P$, $-D$ is ample on the fibers of the projection of $\Kgnb{0,0}(G(2,4),2)$ to the Chow variety and $D$ is ample on the fibers of the projection of the Hilbert scheme to the Chow variety. We conclude that $Hilb_{2x-1}(G(2,4))$ is the flip of $\Kgnb{0,0}(G(2,4), 2)$ over the Chow variety.
The rest of the Theorem follows from the previous lemma and the definition of $P$.
\end{proof}
\smallskip

\begin{remark}
For $G(k,n)$ with $(k,n) \not= (2,4)$, the flip of $\Kgnb{0,0}(G(k,n),2)$ corresponding to a divisor $D$ in the domain bounded by $F, \Hh$ and $\HH$ is no longer the Hilbert scheme, but a divisorial contraction of the Hilbert scheme.

\end{remark}

\bigskip

\section{Degree three maps to Grassmannians} \label{sec-three}
Let $3 \leq k < k+3 \leq n$. In this section,  we study the stable base locus decomposition of $\Kgnb{0,0}(G(k,n),3)$. We begin by introducing the divisor classes that will play an important role in our discussion.
The Neron-Severi space is spanned by the divisors $\Hh, \HH$ and $\Delta=\Delta_{1,2}$. In this basis, the divisors $\ddeg, \dunb$ and $T$ have the following expressions (see \S 4 and 5 of \cite{coskun:grass}):

\begin{eqnarray*}
T &=& \frac{2}{3}( \Hh +  \HH + \Delta) \\
D_{deg} &=& \frac{1}{3} (- \Hh + 2 \HH - \Delta) \\
D_{unb} &=& \frac{1}{3} ( 2 \Hh - \HH - \Delta)
\end{eqnarray*}

\begin{definition}
Let $N= {n \choose k}-1$. The Pl\"{u}cker map embeds $G(k,n)$ in $\PP^N$.
Let $\ccc[\lambda]$ denote the closure of the locus of maps $f$ in $\Kgnb{0,0}(G(k,n),3)$ with irreducible domain such that the map $f$ factors through the inclusion of some Schubert variety $\Sigma_{\lambda}$ in $G(k,n)$. Let $\qqq(\lambda)\lll[\mu]$ denote the closure of the locus of maps with reducible domains $C_1 \cup C_2$ such that $f$ restricts to a degree one map on $C_1$ and a degree two map on $C_2$ such that the image of $f|_{C_2}$ is contained in some Schubert variety $\Sigma_{\lambda}$ and the entire image of $f$ is contained in some Schubert variety $\Sigma_{\mu}$.
\end{definition}

\begin{example}
Let $f \in \Kgnb{0,0}(G(k,n),3)$ be a general stable map. Then the linear spaces parameterized by the image of $f$ sweep out a cubic scroll $S_{0, \dots, 0, 1, 1, 1}$. Such a scroll is a cone over the Segre embedding of $\PP^1 \times \PP^2$. In particular, every stable map in $\Kgnb{0,0}(G(k,n),3)$ lies in a Schubert variety of the form $\Sigma_{(3,2,1)^*}$.  Hence, $\ccc[(3,2,1)^*] = \Kgnb{0,0}(G(k,n),3)$.
$\ccc[(1)^*]$ is the locus of maps in $\Kgnb{0,0}(G(k,n),3)$ that are triple covers of a line in $G(k,n)$. $\qqq\lll[(3,2,1)^*]$ is the boundary divisor.
\end{example}

We begin by introducing two divisor classes $P$ and $F$ that are pull-backs of divisors by the natural morphism $$i: \Kgnb{0,0}(G(k,n), 3) \rightarrow \Kgnb{0,0}(\PP^N,3).$$ In each case, we will calculate the divisor class and determine its base locus.
\smallskip

\noindent {\bf The class $P$.} The image of a map of degree three spans a linear space of dimension three or less in the Pl\"{u}cker embedding of $G(k,n)$. Consider the Zariski open set $U$ in $\Kgnb{0,0}(G(k,n),3)$ where the linear span of the image of $f$ is $\PP^3$.  Let $P$ denote the class of the closure (in $\Kgnb{0,0}(G(k,n),3)$) of the locus in $U$ where the span of $f$ intersects a fixed $\PP^{N-4}$.

\begin{lemma}\label{disc-P}
The class $P$ is given by
$$P = \frac{1}{3} (2\Hh + 2\HH - \Delta).$$ The stable base locus of $P$ consists of $\ccc[(1,1)^*]\cup \ccc[(2)^*] \cup \qqq((1)^*)\lll$.
\end{lemma}
\begin{proof}
Let $B_1$ (respectively, $B_2$) denote the curve in $\Kgnb{0,0}(G(k,n),3)$ induced by a pencil of twisted cubic curves on a quadric surface contained in a fixed $\Sigma_{(1,1,1)^*}$ (respectively, $\Sigma_{(3)^*}$). These pencils sweep out a surface with cohomology class $2\sigma_{(1,1)^*}$ (respectively, $2\sigma_{(2)^*}$) and have four reducible members.  Let $B_3$ denote the curve in $\Kgnb{0,0}(G(k,n),3)$ induced by a pencil of lines in $\Sigma_{(1,1)^*}$ union a fixed conic attached at the base-point of the pencil. The class $P$ is determined by the following intersection numbers.
$$B_1 \cdot \Hh = 2, \ \ B_1 \cdot \HH = 0, \ \ B_1 \cdot \Delta = 4, \ \ \ \ B_1 \cdot P = 0  $$
$$B_2 \cdot \Hh = 0, \ \ B_2 \cdot \HH = 2, \ \ B_2 \cdot  \Delta = 4, \ \ \ \ B_2 \cdot P = 0  $$
$$B_3 \cdot \Hh = 1, \  \ B_3 \cdot \HH = 0, \ \ B_3 \cdot \Delta = -1, \ \ B_3 \cdot P = 1 $$
$P$ is the pull-back of a very ample divisor class under the rational map $$\phi_P: \Kgnb{0,0}(G(k,n),3) \dashrightarrow \GG(3,N)$$ mapping a stable map to the span of its image. Hence, the base locus of $P$ is contained in the indeterminacy locus of the map $\phi_P$. Namely, it is contained in either the locus of maps whose (reduced) image is a curve of degree less than three or a curve of degree three that spans a $\PP^2$. If a curve of degree three in the Grassmannian spans a $\PP^2$, then the $\PP^2$ must be contained in the Grassmannian since the ideal of the Grassmannian in its Pl\"{u}cker embedding is generated by quadrics. We conclude that the base locus of $P$ is contained in the locus $\ccc[(1,1)^*] \cup \ccc[(2)^*] \cup \qqq((1)^*)\lll$ (note that the locus of three-to-one maps onto a line is contained in $\ccc[(1,1)^*] \cap \ccc[(2)^*]$). Conversely, let $B_4$ (respectively, $B_5$) be the curves in $\Kgnb{0,0}(G(k,n),3)$ induced by a pencil of nodal cubics in $\Sigma_{(1,1)^*}$ (respectively, $\Sigma_{(2)^*}$) containing a fixed node and 5 base-points. Curves in the classes $B_4$ and $B_5$ cover the loci $\ccc[(1,1)^*]$ and $\ccc[(2)^*]$, respectively. We have the following intersection numbers
$$B_4 \cdot \Hh = 1, \ \ B_4 \cdot \HH= 0, \ \ B_4 \cdot \Delta = 5, \ \ B_4 \cdot P = -1$$
$$B_5 \cdot \Hh = 0, \ \ B_5 \cdot \HH= 1, \ \ B_5 \cdot \Delta = 5, \ \ B_5 \cdot P = -1$$
Therefore, $\ccc[(1,1)^*] \cup \ccc[(2)^*] $ must be contained in the base locus of $P$.
Similarly, let $B_6$ be a moving curve in $\qqq((1)^*)\lll$ such that the (reduced) image of the maps parameterized by $B_6$ is a fixed pair of lines in $G(k,n)$. Then since $B_6 \cdot \Hh= B_6 \cdot \HH = 0$ and $B_6 \cdot \Delta > 0$, $B_6 \cdot P < 0$. We conclude that $\qqq((1)^*)\lll$ is contained in the base locus of $P$.
\end{proof}
\smallskip

\noindent {\bf The class $F$.} Fix two linear spaces $\Lambda \cong \PP^{N-3} \subset \Gamma \cong \PP^{N-1}$ in $\PP^N$. Let $V$ denote the open subset of $\Kgnb{0,0}(G(k,n),3)$ parameterizing maps $f$ such that $f^{-1}(\Gamma)$ is three distinct points. Let $F$ denote the class of the closure in $\Kgnb{0,0}(G(k,n),3)$ of the locus of maps $f$ such that the line $l$ spanned by a pair of points in $\Gamma \cap \mbox{Image}(f)$ intersects $\Lambda$. Equivalently, the projection from $\Lambda$ of the image of $f$ has a node contained in the   image of the projection of $\Gamma$.

\begin{lemma}\label{disc-F}
The class $F$ is equal to $$F= \frac{1}{3}(5 \Hh + 5 \HH - \Delta)$$ The stable base locus of $F$ consists of $\ccc[(1)^*] \cup \qqq((1)^*)\lll$.
\end{lemma}
\begin{proof}
Let $B_4$ and $B_5$ be the curves introduced in the proof of Lemma \ref{disc-P}. Then $F$ is the class of the pull-back under the natural inclusion of the corresponding divisor class from $\Kgnb{0,0}(\PP^N, 3)$. By \cite{dawei:cubic}, $F \cdot B_4 = F \cdot B_5 = 0$. On the other hand, $B_3 \cdot F = 2$. The formula for $F$ follows from these intersection numbers and the calculations in the proof of Lemma \ref{disc-P}. Suppose the image of a map $f$ is a curve of degree three in $G(k,n) \subset \PP^N$. Then we can always choose a hyperplane $\Gamma$ in $\PP^N$ that intersects the image of $f$ in three distinct points $p_1, p_2, p_3$. By taking $\Lambda$ to be a codimension two linear subspace of $\Gamma$ not intersecting the lines  joining any pair of points $p_i, p_j$ with $i \not=j$, we obtain a divisor in the class $F$ whose support does not contain $f$. We conclude that the base locus of $F$ is contained in $\ccc[(1)^*] \cup \qqq((1)^*)\lll$. Since the curve $B_6$ introduced in the proof of Lemma \ref{disc-P} has $B_6 \cdot F <0$, $\qqq((1)^*)\lll$ is in the base locus of $P$. Similarly, let $B_7$ be a moving one-parameter family of maps in $\ccc[(1)^*]$ intersecting the boundary divisor whose image is a fixed line in $G(k,n)$.  Since $B_7 \cdot \Hh = B_7 \cdot \HH = 0$ and $B_7 \cdot \Delta >0$, we conclude that $B_7 \cdot F <0$ and $\ccc[(1)^*]$ is in the base locus of $F$.
\end{proof}

The $\PP^{k-1}$'s parameterized by a twisted cubic curve in $G(k,n)$ sweep out a rational scroll of degree three in $\PP^{n-1}$. We can define divisors in $\Kgnb{0,0}(G(k,n),3)$ by imposing conditions on this scroll.
\smallskip

\noindent {\bf The classes $S$ and $S'$.} Fix a linear space $\Lambda = \PP^{n-k-2} \subset \Gamma = \PP^{n-k}$ in $\PP^{n-1}$. Let $U$ be the Zariski open subset of $\Kgnb{0,0}(G(k,n),3)$ consisting of maps $f$ such that the linear spaces parameterized by the image of $f$ sweep out a balanced cubic rational scroll (equivalently, a cone over the Segre embedding of $\PP^1 \times \PP^2$). Let $S$ be the class of the closure in $\Kgnb{0,0}(G(k,n),3)$ of the locus of maps where the scroll contains a quadric hypersurface of dimension $k-1$ whose span contains $\Lambda$ and intersects $\Gamma$ in a linear space of dimension $n-4$. The projection of a cubic scroll of dimension $k$ in $\PP^{n-1}$ from $\Lambda$ is a cubic hypersurface in $\PP^{k+1}$ which is double along a $\PP^{k-1}$. Conversely, an irreducible cubic hypersurface in $\PP^{k+1}$ which is double along a $\PP^{k-1}$ arises as a projection of a cubic scroll. $S$ is the class of the divisor of maps where the singular locus of the projection of the scrolls from $\Lambda$ intersects a fixed line (the image of the projection of $\Gamma$). Given a divisor class $D$ in $\Kgnb{0,0}(G(n-k,n),3)$ we can define a dual divisor class $D'$ in $\Kgnb{0,0}(G(k,n),3)$. $G(k,n)$ and $G(n-k, n)$ are isomorphic. This isomorphism induces an isomorphism between $\Kgnb{0,0}(G(k,n),3)$ and $\Kgnb{0,0}(G(n-k,n),3)$. Let $D'$ denote the pull-back of a divisor class $D$ in $\Kgnb{0,0}(G(n-k,n),3)$ under the isomorphism. In particular, define $S'$ to be the divisor class obtained by starting with $S$ in $\Kgnb{0,0}(G(n-k, n),3)$.

\begin{lemma}\label{disc-S}
The classes $S$ and $S'$ are equal to the following:
$$S = \frac{1}{3}(-\Hh + 5\HH-\Delta)$$
$$S' = \frac{1}{3}(\ 5\Hh - \HH - \Delta)$$
The stable base locus of $S$ consists of $\ccc[(1,1,1)^*] \cup \qqq((1,1)^*)\lll$. The stable base locus of $S'$ consists of $\ccc[(3)^*] \cup \qqq((2)^*)\lll$.
\end{lemma}

\begin{proof}
The assertions about $S'$ follow from the assertions about $S$ by duality. To calculate the class of $S$, we use test families. Consider a pencil of plane cubics with a fixed node. Take the cone over this pencil with a vertex equal to a projective linear space  $\PP^{k-2}$. This pencil of cubic scrolls induces a one-parameter family of degree three curves in $G(k,n)$, hence a curve $B_8 \in \Kgnb{0,0}(G(k,n),3)$. The following intersection numbers are easy to calculate:
$$B_8 \cdot \Hh =0, \ \ B_8 \cdot \HH = 1, \ \ B_8 \cdot \Delta = 5, \ \ B_8 \cdot S = 0$$
Let $B_9$ (respectively, $B_{10}$)  be the curves induced in $\Kgnb{0,0}(G(k,n),3)$ by attaching a conic at the base point of a pencil of lines contained in $\Sigma_{(1,1)^*}$ (respectively, $\Sigma_{(2)^*}$). We can interpret the corresponding scrolls as follows. The scrolls swept out by the linear spaces parameterized by points in $B_9$, are the union of a fixed quadric scroll with a fixed linear space $L$ of projective dimension $k$ having a common $\PP^{k-1}$. The intersection of the $\PP^{k-1}$'s varying in a pencil in $L$ is the only data that varies. The scrolls swept out by the linear spaces parameterized by points in $B_{10}$, are the unions of a fixed quadric scroll with a pencil of linear spaces of projective dimension $k$ having a common $\PP^{k-1}$ with the quadric scroll. Using these geometric descriptions, the following intersection numbers are straightforward to calculate:
$$B_9 \cdot \Hh = 1, \ \ B_9 \cdot \HH = 0, \ \ B_9 \cdot \Delta = -1, \ \ B_9 \cdot S = 0$$
$$B_{10} \cdot \Hh = 0, \ B_{10} \cdot \HH = 1, \ B_{10} \cdot \Delta = -1, \ B_{10} \cdot S= 2$$ The class of $S$ (and by duality that of $S'$) follows from these calculations.

Let $B_{11}$ be the curve induced in $\Kgnb{0,0}(G(k,n),3)$ from a pencil of conics in $\Sigma_{(1,1)^*}$ union a line at a base point of the pencil. Curves in the same class as $B_{11}$ cover the locus $\qqq((1,1)^*)\lll$. Since
$$B_{11} \cdot \Hh = 1, \ \ B_{11}\cdot \HH = 0, \ \ B_{11} \cdot \Delta = 2,$$ we conclude that $B_{11} \cdot S = -1 < 0$. Therefore, $\qqq((1,1)^*)\lll$ is in the stable base locus of $S$.
Recall from the proof of Lemma \ref{disc-P} that  $B_{1}$ is a pencil of twisted cubics on a quadric contained in $\Sigma_{(1,1,1)^*}$. Since curves in the class $B_1$ cover $\ccc[(1,1,1)^*]$ and $B_1 \cdot S = -2 < 0$, we conclude that  $\ccc[(1,1,1)^*]$ is in the stable base locus of $S$.

Suppose $X$ is an irreducible cubic scroll of dimension $k$ that spans a projective linear space of dimension $k+1$ or more. Then the singular locus of $X$ has dimension less than or equal to $k-1$. We can always choose a linear space  $\Lambda = \PP^{n-k-2}$ so that the projection of $X$ still spans a linear space of dimension $k+1$. We can then find a line $l$ disjoint from the singular locus of the projection $X$. If we take $\Gamma$ to be the span of $\Lambda$ and $l$, we obtain a section of $S$ not vanishing on the point in $\Kgnb{0,0}(G(k,n),3)$ induced by the scroll $X$. Similarly, as long as the linear spaces parameterized by a point in $\Kgnb{0,0}(G(k,n),3)$ do not cover a linear space of dimension $k$ multiple-to-one, then the singular locus of the resulting (possibly reducible) scroll has dimension less than or equal to $k-1$ and the same argument shows that the point is not in the base locus of $S$. We conclude that the base locus of $S$ is contained in $\ccc[(1,1,1)^*] \cup \qqq((1,1)^*)\lll$, hence equality holds. This completes the proof of the proposition.
\end{proof}

Finally, the following lemma determines the stable base locus of $\ddeg$ and $\dunb$.

\begin{lemma}\label{deg-unb}
The stable base locus of $\dunb$ is $\ccc[(3,2)^*] \cup \qqq((2)^*)\lll$. The stable base locus of $\ddeg$ is $\ccc[(2,2,1)^*] \cup \qqq((1,1)^*)\lll$.
\end{lemma}
\begin{proof}
By duality, it suffices to consider the stable base locus of $\ddeg$. Note that $\ccc[(2,2,1)^*] \cup \qqq((1,1)^*)\lll$ is the locus of maps whose images lie in a sub-Grassmannian $G(k, k+2)$ of $G(k,n)$. Suppose the image of a map $f$ does not lie in a sub-Grassmannian $G(k,k+2)$, then the image of $f$ lies in a sub-Grassmannian $G(k,k+3)$. Take a linear space $\Lambda$ of dimension $n-k-3$  that does not intersect the $(k+3)$-dimensional linear space spanned by the linear spaces parameterized by the image of $f$. The locus of maps $g \in \Kgnb{0,0}(G(k,n),3)$ such that the projection from $\Lambda$ of the span of the linear spaces parameterized by $g$ has dimension less than or equal to $k+2$ is an effective divisor $D$ in the class $\ddeg$. Since $f \not\in D$, we conclude that the stable base locus of $\ddeg$ is contained in $\ccc[(2,2,1)^*] \cup \qqq((1,1)^*)\lll$. By the argument in the previous lemma, $\qqq((1,1)^*)\lll$ is in the stable base locus of $\ddeg$. To see that $\ccc[(2,2,1)^*]$ is contained in the stable base locus of $\ddeg$, take a pencil of cubic hypersurfaces double along a fixed projective linear space $\PP^{k-1}$. (Note that a general projection of a cubic scroll of dimension $k$ to $\PP^{k+1}$ is a cubic hypersurface double along a $\PP^{k-1}$.) This family of cubic hypersurfaces induces a curve in $\Kgnb{0,0}(G(k,n),3)$ whose intersection number with $\ddeg$ is $-1$. These curves cover the locus $\ccc[(2,2,1)^*]$. We conclude that the stable base locus of $\ddeg$ is $\ccc[(2,2,1)^*]\cup \qqq((1,1)^*)\lll$.
\end{proof}
\smallskip

\begin{definition}
Let $U$ and $U'$ be the divisor classes $$U= 2\Hh + 5\HH - \Delta, \ \ \ U' = 5\Hh + 2 \HH - \Delta.$$ Let $R$ be the divisor class $$R= \Hh + \HH - \Delta.$$ Let $V$ and $V'$ be the divisor classes $$V= \Hh + 4\HH -2 \Delta, \ \ \ V' = 4 \Hh + \HH-2\Delta.$$
\end{definition}

The next theorem describes the stable base locus decomposition of  $\Kgnb{0,0}(G(k,n), 3)$. Since there are 22 chambers in the decomposition, the statement of the theorem is necessarily long. The decomposition is summarized in Figure \ref{fig:cubdecomp}, where we draw a cross-section of the three-dimensional cone.

\begin{figure}[!ht] 
   \centering
   \includegraphics[width=5.5in]{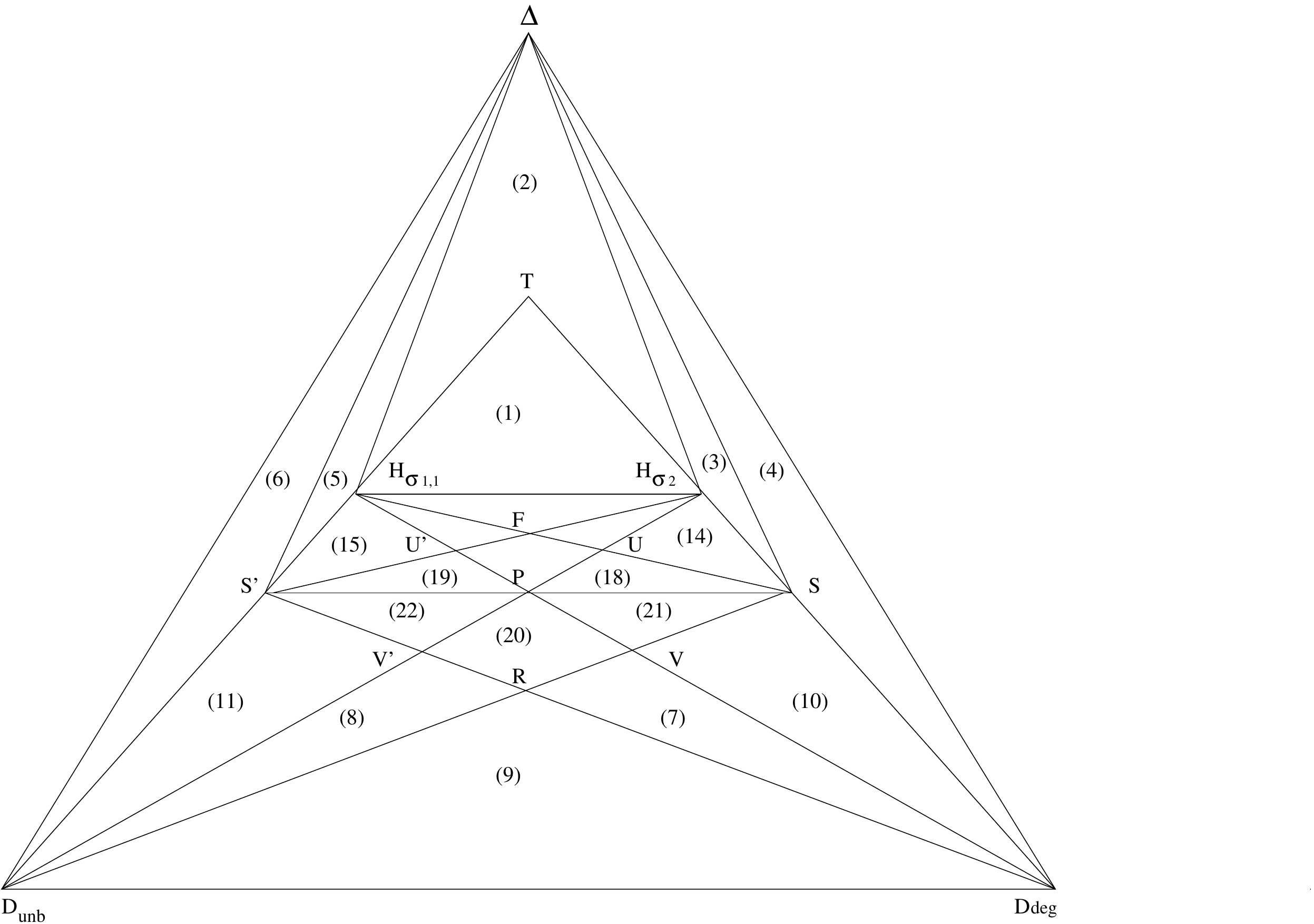}
   \caption{The stable base locus decomposition of $\Kgnb{0,0}(G(k,n),3)$.}
   \label{fig:cubdecomp}
\end{figure}

\begin{theorem}\label{de-three}
The stable base locus decomposition of the effective cone of $\Kgnb{0,0}(G(k,n), 3)$, with $3 \leq k < k+3 \leq n$,  is as follows:
\begin{enumerate}
\item In the closed cone spanned by  $\Hh, \HH $ and $T$ the stable base locus is empty.
\smallskip

\item In the domain bounded by $\Delta, \Hh, T$ and $\HH$ union $c( \Hh \overline{\Delta}) \cup c(\HH \overline{\Delta})$ the stable base locus is equal to the boundary divisor.
\smallskip

\item In the domain bounded by $\HH, \Delta$ and $S$ union $c( \Delta S)$,  the stable base locus is the union of $\ccc[(1,1, 1)^*]$ and the boundary divisor.
\smallskip

\item In the domain bounded by $\ddeg, S$ and $\Delta$ union $c(\Delta \ddeg)$, the stable base locus is the union of $\ccc[(2,2,1)^*]$ and the boundary divisor.
\smallskip

\item In the domain bounded by $\Hh, \Delta$ and $S'$  union $c(\Delta S')$, the stable base locus is the union of $\ccc[(3)^*]$ and the boundary divisor.
\smallskip

\item In the domain bounded by $\dunb, S'$ and $\Delta$ union $c(\Delta \dunb)$ the stable base locus is the union of  $\ccc[(3,2)^*]$ and the boundary divisor.
\smallskip

\item In the domain bounded by $\ddeg, R$ and $V$ union $c(\ddeg R)$, the stable base locus is $\ccc[(2,2,1)^*] \cup \ccc[(3)^*] \cup \qqq((2)^*)\lll \cup \qqq((1,1)^*)\lll$.
\smallskip

\item In the domain bounded by $\dunb, R$ and $V'$ union $c(\dunb R)$, the stable base locus is $\ccc[(3,2)^*] \cup \ccc[(1,1,1)^*] \cup \qqq((1,1)^*)\lll \cup \qqq((2)^*)\lll$.
\smallskip

\item In the domain bounded by $\ddeg, \dunb$ and $R$ union $c(\ddeg \dunb)$, the stable base locus is the union  $\ccc[(3,2)^*] \cup \ccc[(2,2,1)^*] \cup \qqq((2)^*)\lll \cup \qqq((1,1)^*)\lll$.
\smallskip

\item In the domain bounded by $\ddeg, S$ and $V$ union $c(V \overline{D}_{deg}) \cup c(S \overline{D}_{deg})$, the stable base locus is the locus  $\ccc[(2,2,1)^*] \cup \qqq((1,1)^*)\lll$.
\smallskip

\item In the domain bounded by $\dunb, S'$ and $V'$ union $c( V' \overline{D}_{unb}) \cup c(S' \overline{D}_{unb} )$, the stable base locus is the locus $\ccc[(3,2)^*] \cup \qqq((2)^*)\lll$.
\smallskip

\item In the domain bounded by $F, \Hh$ and $\HH$ union $c(\Hh \overline{F}) \cup c(\HH \overline{F})$ the stable base locus is $\ccc[(1)^*] \cup \qqq((1)^*)\lll$.
\smallskip

\item  In the domain bounded by $P, U, F$ and $F'$ union $c(U \overline{P}) \cup c(U' \overline{P} )$, the stable base locus is $\ccc[(1,1)^*] \cup \ccc[(2)^*] \cup \qqq((1)^*)\lll$.
\smallskip

\item In the domain bounded by $S, \HH$ and $U$ union $c(U \overline{S}) \cup c(\HH \overline{S})$, the stable base locus is

\noindent $\ccc[(1,1, 1)^*] \cup \qqq((1)^*)\lll$.
\smallskip

\item In the domain bounded by $S', \Hh$ and $U'$ union $c(U' \overline{S}' ) \cup c(\Hh \overline{S}')$, the stable base locus is $\ccc[(3)^*] \cup \qqq((1)^*)\lll$.
\smallskip

\item In the domain bounded by $F, \HH$ and $U$ union $c(UF) \cup c(U\HH)$, the stable base locus is $\ccc[(1,1)^*] \cup \qqq((1)^*)\lll$.
\smallskip

\item In the domain bounded by $F, \Hh$ and $U'$ union $c(U'F) \cup c(U'\Hh)$, the stable base locus is $\ccc[(2)^*] \cup \qqq((1)^*)\lll$.
\smallskip

\item In the domain bounded by $P, S$ and $U$ union $c(PS)$, the stable base locus is $\ccc[(1,1,1)^*] \cup \ccc[(2)^*] \cup \qqq((1,1)^*)\lll$.

\smallskip

\item In the domain bounded by $P, S'$ and $U'$ union $c(PS')$, the stable base locus is $\ccc[(1,1)^*] \cup \ccc[(3)^*] \cup \qqq((2)^*)\lll$.
\smallskip

\item In the domain bounded by $P, V, R$ and $V'$, the stable base locus is $\ccc[(1,1,1)^*] \cup \ccc[(2,2)^*] \cup \ccc[(3)^*] \cup \qqq((1,1)^*)\lll \cup \qqq((2)^*)\lll$.
\smallskip

\item In the domain bounded by $P, S$ and $V$ the stable base locus is $\ccc[(1,1,1)^*] \cup \ccc[(2,2)^*]\cup \qqq((1,1)^*)\lll$.
\smallskip

\item In the domain bounded by $P, S'$ and $V'$, the stable base locus is $\ccc[(3)^*] \cup \ccc[(2,2)^*] \cup \qqq((2)^*)\lll$.
\smallskip

\end{enumerate}
\end{theorem}

\begin{proof} The divisors $\Hh, \HH$ and $T$ are base-point free (see \cite{coskun:grass} \S 5). It follows that in the closed cone generated by these divisors the stable base locus is empty. Let $D= a\Hh + b \HH + c \Delta$ be an effective divisor. If curves with class $B$ cover a subvariety $X$ of $\Kgnb{0,0}(G(k,n),3)$ and $B \cdot D < 0$, then $X$ has to be contained in the base locus of $D$. Conversely, if $D$ can be expressed as a non-negative linear combination of $D'$ and base-point-free divisors, then the stable base locus of $D$ is contained in the stable base locus of $D'$. Using these two observations repeatedly we can determine the stable base locus decomposition. \smallskip

\noindent $\bullet$ The boundary. Recall from the proof of Lemma \ref{disc-P} that $B_3$ is the class of the curve in $\Kgnb{0,0}(G(k,n),3)$ induced by a pencil of lines in $\Sigma_{(1,1)^*}$ union a fixed conic attached at the base point. $B_3 \cdot D = a-c$.
Similarly, let $B_{12}$ be the class of the curve  in $\Kgnb{0,0}(G(k,n), 3)$ induced by a pencil of lines in $\Sigma_{(2)^*}$ union a fixed conic. $B_{12} \cdot D = b-c$. Since curves in the class $B_3$ and $B_{12}$ cover the boundary divisor, we conclude that the boundary divisor is in the stable  base locus of $D$ whenever $a<c$ or $b<c$.  Equivalently, the boundary divisor is in the stable base locus of
$D$ if $D$ is in the complement of the closed cone generated by $\dunb, \ddeg$ and $T$. Conversely, since $T$ is base point free, the stable base locus of  any divisor in  the closed cone generated by $\dunb, \ddeg$ and $T$ must be contained in the union of  the stable base loci of $\ddeg$ and $\dunb$. Therefore, the boundary divisor is contained in the stable base locus of $D$ if and only if $D$ is in the complement of the closed cone generated by $\dunb, \ddeg$ and $T$.
\smallskip

\noindent $\bullet$ $\qqq((1,1)^*)\lll$ and $\qqq((2)^*)\lll$. Recall from the proof of Lemma \ref{disc-S} that $B_{11}$ is the curve in $\Kgnb{0,0}(G(k,n),3)$ obtained by taking a pencil of conics in $\Sigma_{(1,1)^*}$ union a fixed line at a base point of a pencil. Since curves in the same class as $B_{11}$ cover $\qqq((1,1)^*) \lll$ and $B_{11} \cdot D = a + 2c$, we conclude that $\qqq((1,1)^*) \lll$ is in the stable base locus of $D$ if $a< -2c$. By replacing $\Sigma_{(1,1)^*}$ with $\Sigma_{(2)^*}$ in this discussion,   we conclude that $\qqq((2)^*) \lll$ is in the stable base locus of $D$ if $b<-2c$.  Conversely, $\qqq((1,1)^*)\lll$ (resp., $\qqq((2)^*)\lll$)  is not contained in the stable base locus of $\dunb$ (resp., $\ddeg$). We conclude that $\qqq((1,1)^*)\lll$ is contained in the stable base locus of $D$ if and only if $D$ is in the complement of the closed cone spanned by $\dunb, \HH$ and $T$. Similarly, $\qqq((2)^*)\lll$ is in the stable base locus of $D$ if and only if $D$ is in the complement of $\ddeg, \Hh$ and $T$.
\smallskip

\noindent $\bullet$ $\qqq((1)^*)\lll$. During the proof of Lemma \ref{disc-P}, we showed that $\qqq((1)^*)\lll$ is in the stable base locus of $D$ if $c<0$.  It follows that $\qqq((1)^*)\lll$ is in the stable base locus of $D$ if and only if $D$ is in the complement of the closed cone generated by $\Hh, \HH$ and $T$.
\smallskip

\noindent $\bullet$ $\ccc[(3,2)^*]$ and $\ccc[(2,2,1)^*]$. We would like to show that if $D$ is an effective divisor in the complement of the closed cone generated by $\dunb, S, \Delta$ (respectively, in the complement of the closed cone generated by $\ddeg, S', \Delta$), then $\ccc[(2,2,1)^*]$ (respectively, $\ccc[(3,2)^*]$) is in the stable base locus of $D$.  We define two families of cubic surface scrolls in $\PP^4$. Fix a pencil of conics in $\PP^2$ and a general line $l$ in $\PP^4$. Fix three points $p_1, p_2, p_3$ on the line and three of the base points $q_1, q_2, q_3$ of the pencil of conics. For each member $C_t$ of the pencil of conics, there exists a unique cubic scroll containing $C_t, l$ and the lines $l_{p_i, q_i}$ joining $p_i$ to $q_i$. Let $\mathcal{F}_1$ be the corresponding family of cubic scrolls. Note that $\mathcal{F}_1$ has three reducible members; all of the scrolls in $\mathcal{F}_1$ are non-degenerate; and the directrix of the scrolls $l$ does not vary in the family. Take the cone with a fixed vertex $\PP^{k-3}$ to obtain a family of cubic scrolls of dimension $k$. This family induces a curve with class $B_{13}$ in $\Kgnb{0,0}(G(k,n), 3)$. We claim that
$$B_{13} \cdot \Hh= 1, \ \ B_{13} \cdot \HH= 2, \ \ B_{13} \cdot \Delta = 3$$
The last equality is clear since the family has three reducible elements.
The family $\mathcal{F}_1$ induces a curve with class $B_{13}'$ in $\Kgnb{0,0}(G(2,5),3)$. It is straightforward to see that $B_{13} \cdot \Hh = B_{13}' \cdot \Hh'$ and $B_{13} \cdot \HH = B_{13}' \cdot \HH '$, where the primes denote that the intersection is taking place in $\Kgnb{0,0}(G(2,5),3)$. Since in the family $\mathcal{F}_1$ the members are non-degenerate and the directrices are constant, we have the intersection numbers $B_{13}' \cdot \ddeg ' = B_{13}' \cdot \dunb ' = 0$ in $\Kgnb{0,0}(G(2,5),3)$. The classes of these divisors are calculated in \cite{coskun:grass} (see also the next section). Solving for the coefficients we obtain the claimed equalities.
Next, take a general projection of the scroll $S_{2,2}$ to $\PP^4$.  Recall that the scroll $S_{2,2}$ is the embedding of $\PP^1 \times \PP^1$ in $\PP^5$ under the complete linear system $\OO_{\PP^1 \times \PP^1}(1,2)$. Take a general line $l$ in $\PP^4$ and fix an isomorphism between $l$ and the conics in $S_{2,2}$ and let $S_{1,2,2}$ be the scroll generated by taking the spans of the points under this isomorphism. The scroll $S_{1,2,2}$ gives rise to a one-parameter family $\mathcal{F}_2$ of cubic scrolls in $\PP^4$.  In the family $\mathcal{F}_2$ none of the members are reducible and the directrices of all the cubic scrolls are $l$. Taking the cone over $\mathcal{F}_2$ with a fixed vertex $\PP^{k-3}$ induces a curve with class $B_{14}$ in $\Kgnb{0,0}(G(k,n),3)$. We have the following intersection numbers:
$$B_{14} \cdot \Hh = 1,  \ \ B_{14} \cdot \HH = 5, \ \ B_{14} \cdot \Delta = 0$$
Since the degree of the cone over $S_{1,2,2}$ is five and the family does not have any reducible elements, the last two equalities are immediate. The first equality can be computed, as in the previous case, by noting that $\mathcal{F}_2$ induces a curve $B_{14}'$ in $\Kgnb{0,0}(G(2,5),3)$ satisfying the equalities $B_{14} \cdot \Hh = B_{14}' \cdot \Hh$ and $B_{14}' \cdot \dunb' = 0$.
Since curves in the class $B_{13}$ and $B_{14}$ cover the locus $\ccc[(3,2)^*]$, we conclude that if the effective divisor $D$ satisfies $a+5b<0$ or $a + 2b + 3c<0$, then $\ccc[(3,2)^*]$ is in the stable base locus of $D$. By duality, if $5a+b<0$ or $2a+b+3c<0$, then the locus $\ccc[(2,2,1)^*]$ is in the stable base locus of $D$. Conversely, by Lemmas \ref{disc-S} and \ref{deg-unb}, $\ccc[(2,2,1)^*]$ is not contained in the union of the stable base loci of  $\dunb, S$ and $\Delta$ and $\ccc[(3,2)^*]$ is not contained in the union of the stable base loci of $\ddeg, S'$ and $\Delta$. We conclude that  $\ccc[(2,2,1)^*]$ (respectively, $\ccc[(3,2)^*]$) is in the stable base locus of $D$ if and only if $D$ is an effective divisor in the complement of the closed cone generated by $\dunb, S, \Delta$ (respectively, in the complement of the closed cone generated by $\ddeg, S', \Delta$).
\smallskip

\noindent $\bullet$ $\ccc[(2,2)^*]$. Fix a linear space $\Lambda$ of dimension $k-2$ disjoint from a four dimensional linear space $\Gamma$. Let $\phi: G(2,4) \rightarrow G(k,n)$ be the morphism obtained by taking the span of any two dimensional linear space in $\Gamma$ with $\Lambda$ and considering the resulting subspaces as a subspace of the $n$-dimensional ambient vector space. A codimension two linear section of $G(2,4)$ in its Pl\"{u}cker embedding maps to a quadric surface in the Pl\"{u}cker embedding of $G(k,n)$ of class $\sigma_{(1,1)^*} + \sigma_{(2)^*}$. Consider a pencil of twisted cubics on this quadric surface and let $B_{15}$ be its class. The following intersection numbers are easy to compute:

$$B_{15} \cdot \Hh=1, \ \  B_{15} \cdot \HH = 1, \ \ B_{15} \cdot \Delta = 4$$
Since curves with class $B_{15}$ cover the locus $\ccc[(2,2)^*]$, we conclude that $\ccc[(2,2)^*] $ is in the base locus of $D$ if $a+b+4c<0$. On the other hand, $\ccc[(2,2)^*]$ is not in the union of the stable base loci of $S, S'$ and $\Delta$. We conclude that $\ccc[(2,2)^*]$ is in the base locus of $D$ if and only if $D$ is in the complement of the closed cone generated by $S, S'$ and $\Delta$.
\smallskip

\noindent $\bullet$ $\ccc[(1,1,1)^*]$ and $\ccc[(3)^*]$. Recall from the proof of Lemma \ref{disc-P} that $B_1$ (respectively, $B_2$) are the classes of the curves in $\Kgnb{0,0}(G(k,n),3)$ induced by a pencil of twisted cubics on a quadric surface contained in $\Sigma_{(1,1,1)^*}$ (respectively, $\Sigma_{(3)^*}$). Curves in the class $B_1$ (respectively, $B_2$) cover $\ccc[(1,1,1)^*]$ (respectively, $\ccc[(3)^*]$). Since $B_1 \cdot D = 2a+4c$ and $B_2 \cdot D = 2b + 4c$, we conclude that if $a<-2c$ (respectively, $b<-2c$), then $\ccc[(1,1,1)^*]$ (respectively, $\ccc[(3)^*]$) is in the stable base locus of $D$. Next, consider a general  projection of the third Veronese embedding of $\PP^2$ in $\PP^9$ to $\PP^3$. The image of a pencil of lines in $\PP^2$ under this map gives rise to a one-parameter family $\mathcal{F}$ of rational cubics in $\PP^3$. Let $B_{16}$ (respectively, $B_{17}$) be the classes of the curves in $\Kgnb{0,0}(G(k,n),3)$ that are induced by taking the family $\mathcal{F}$ in $\Sigma_{(1,1,1)^*}$  (respectively, $\Sigma_{(3)^*}$). The following intersection numbers are easy to compute:  $$B_{16} \cdot \Hh = 9,  \ \ B_{16} \cdot \HH = 0, \  \ B_{16} \cdot \Delta = 0$$
$$B_{17} \cdot \Hh = 0,  \ \ B_{17} \cdot \HH = 9, \  \ B_{17} \cdot \Delta = 0$$
Since curves in the class $B_{16}$ (respectively, $B_{17}$) cover $\ccc[(1,1,1)^*]$ (respectively, $\ccc[(3)^*]$), we conclude that $\ccc[(1,1,1)^*]$ (respectively, $\ccc[(3)^*]$) is in the stable base locus of $D$ if $a<0$ (respectively, $b<0$).  In summary, we conclude that $\ccc[(1,1,1)^*]$ is in the stable base locus of the divisors contained in the complement of the closed cone generated by $ \dunb, \HH$ and $\Delta$. Similarly, $\ccc[(3)^*]$ is in the stable base locus of the divisors contained in the complement of the closed cone generated by $\ddeg, \Hh$ and $\Delta$.
\smallskip

\noindent $\bullet$ $\ccc[(1,1)^*]$ and $\ccc[(2)^*]$. The proof of Lemma \ref{disc-P} shows that if $a + 5c<0$ (respectively, $b+5c <0$), then $\ccc[(1,1)^*]$ (respectively, $\ccc[(2)^*]$) is contained in the base locus of $D$. $\ccc[(1,1)^*]$ is not contained in the union of the stable base loci of $S'$ and $\Delta$. Similarly, $\ccc[(2)^*]$ is not contained in the union of the stable base loci of $S$ and $\Delta$. We conclude that $\ccc[(1,1)^*]$ (resp., $\ccc[(2)^*]$) is in the stable base locus of $D$ if and only if $D$ is in the complement of the closed cone spanned by $S', \HH$ and $\Delta$ (resp., $S, \Hh$ and $\Delta$).
\smallskip

\noindent $\bullet$ $\ccc[(1)^*]$. The proof of Lemma \ref{disc-F} shows that if $c<0$, then the locus of maps that have a component mapping multiple-to-one onto a line is in the stable base locus of $D$. Take a smooth quadric surface in $\Sigma_{(1,1,1)^*}$ or $\Sigma_{(3)^*}$. Fix a three-to-one map from $\PP^1$ to $\PP^1$ and map $\PP^1$ to each member of one of the rulings of the quadric surface. The induced curves $B_{18}$ and $B_{19}$, respectively, have the following intersection numbers:
$$B_{18} \cdot \Hh = 2, \ \ B_{18} \cdot \HH = 0, \ \ B_{18} \cdot \Delta = 0  $$
$$B_{19} \cdot \Hh = 0, \ \ B_{19} \cdot \HH = 2, \ \ B_{19} \cdot \Delta = 0  $$
We conclude that $\ccc[(1)^*]$ is in the stable base locus of $D$ if and only if $D$ is contained in the complement of the closed cone spanned by $\Hh, \HH$ and $\Delta$.
\smallskip

We now combine these observations to conclude the proof of the theorem.
Let $D$ be a divisor contained in the closed cone spanned by $\Delta, \Hh$ and $\HH$ but not contained in the closed cone spanned by $\Hh, \HH$ and $T$. Since $\Hh$ and $\HH$ are base-point-free, the stable base locus of $D$ has to be contained in the boundary divisor. By the discussion of the boundary divisor, the stable base locus of $D$ contains the boundary divisor. We conclude that the stable base locus of $D$ is equal to the boundary divisor. \smallskip

If $D$ is a divisor in the domain bounded by $S, \Delta$ and $\HH$ union $c(\Delta S)$, then the stable base locus of $D$ is contained in the union of the stable base loci of $\Delta$ and $S$. We conclude that the stable base locus of $D$ is $\ccc[(1,1,1)^*]$ union the boundary divisor. Similarly, if $D$ is a divisor in the domain bounded by $S', \Delta$  and $\Hh$ union $c(\Delta S')$, then the stable base locus of $D$ is $\ccc[(3)^*]$ union the boundary divisor.
\smallskip

Suppose $D$ is in the region bounded by $\ddeg, \Delta$ and $S$ union $c(\Delta \ddeg)$ (respectively, in the region bounded by $\dunb, \Delta$ and $S'$ union $c(\Delta \dunb)$). Then the stable base locus of $D$ has to be contained in the union of the stable base loci of $\ddeg$ and $\Delta$ (respectively, $\dunb$ and $\Delta$).  We deduce that in the region bounded by $\ddeg, \Delta$ and $S$ union $c(\Delta \ddeg)$, the stable base locus is equal to $\ccc[(2,2,1)^*]$ union the boundary divisor. In the region bounded by $\dunb, \Delta$ and $S'$ union $c(\Delta \dunb)$, the stable base locus is the union of $\ccc[(3,2)^*]$ and the boundary divisor.
\smallskip

Similarly, if $D$ is in the region bounded by $\ddeg, S$ and $V$ union $c(\overline{D}_{deg} S) \cup c(\overline{D}_{deg} V)$ (respectively, $\dunb, S'$ and $V'$ union $c(\overline{D}_{unb} S') \cup c(\overline{D}_{unb} V')$), then the stable base locus of $D$ has to be a subset of the stable base locus of $\ddeg$ (respectively, $\dunb$).  This follows from the fact that $D$ is a non-negative linear combination of $\ddeg$ (respectively, $\dunb$) and base-point-free divisors $\Hh$ and $\HH$. We conclude that these stable base loci are $\ccc[(2,2,1)^*] \cup \qqq((1,1)^*)\lll$ (respectively, $\ccc[(3,2)^*]\cup \qqq((2)^*)\lll$).
An almost identical argument shows that if $D$ is  in the region bounded by $\ddeg, \dunb$ and $R$ union $c(\ddeg \dunb)$, then the stable base locus of $D$ is $\ccc[(2,2,1)^*] \cup \ccc[(3,2)^*] \cup \qqq((1,1)^*)\lll \cup \qqq((2)^*)\lll$.

\smallskip

If $D$ is in the region bounded by $\ddeg, R$ and $V$ union $c(R \ddeg)$, then the stable base locus of $D$ is contained in the union of the stable base loci of $\ddeg$ and $S'$ since every divisor in this region can be expressed as a non-negative linear combination of $\ddeg, S'$ and base-point-free divisors. Similarly, if $D$ is in the region generated by $\dunb, R$ and $V'$ union $c(R \dunb)$, then the stable base locus is contained in the union of the stable base loci of $S$ and $\dunb$. We conclude that in the region generated by $\ddeg, R$ and $V$ union $c(R \ddeg)$, the stable base locus is $\ccc[(2,2,1)^*] \cup \ccc[(3)^*] \cup  \qqq((2)^*)\lll$. In the region generated by $\dunb, R$ and $V'$ union $c(R \dunb)$, the stable base locus is $\ccc[(3,2)^*] \cup \ccc[(1,1,1)^*] \cup \qqq((1,1)^*)\lll$.
\smallskip

In the region bounded by $F, \HH$ and $\Hh$ union $c(\Hh \overline{F}) \cup c(\HH \overline{F})$, the stable base locus has to be contained in the stable base locus of $F$. We conclude that the stable base locus is $\ccc[(1)^*] \cup \qqq((1)^*)\lll$.
\smallskip

In the domain bounded by $S, U$ and $\HH$ union $c(U \overline{S}) \cup c(\HH \overline{S})$ the stable base locus has to be contained in that of $S$. We conclude that the stable base locus is $\ccc[(1,1,1)^*] \cup \qqq((1,1)^*)\lll$. Similar considerations apply for $S', U'$ and $\Hh$ union $c(U' \overline{S}') \cup c(\Hh \overline{S}')$. The stable base locus in the domain bounded by $F, U$ and $\HH$ union $c(UF) \cup c(U \HH)$ is contained in the stable base locus of $U$ that is contained in the intersection of the stable base loci of $P$ and $S$. We conclude that the stable base locus is $\ccc[(1,1)^*] \cup \qqq((1)^*)\lll$. Similar considerations apply to the domain bounded by $F, U'$ and $\Hh$ union $c(U'F) \cup c(U' \Hh)$.
\end{proof}
\smallskip

\begin{remark}\label{remark-cube}
The proof of Theorem \ref{de-three} also leads to a detailed description of the birational models of $\Kgnb{0,0}(G(k,n),3)$. The model corresponding to $T$ is the moduli space of weighted stable maps $\Kgnb{0,0}(G(k,n),2,1)$ obtained as a divisorial contraction of $\Kgnb{0,0}(G(k,n),3)$ that contracts the boundary divisor. The morphism $\phi_T$ collapses  the locus of maps with reducible domain that have the same degree two component and  the same node, remembering only the degree two component and the point of attachment. For $D \in c(\Hh T)$ or $D \in c(\HH T)$, the models give two other divisorial contractions of the boundary divisor that further admit small contractions to $\Kgnb{0,0}(G(k,n),2,1)$.
The model corresponding to a divisor in $D \in c(\Hh \HH)$ is the normalization of the Chow variety. For such $D$, $\phi_D$ is a small contraction that contracts the locus of maps that have a component multiple-to-one onto their image remembering only the image and the multiplicity.  The normalization of the Chow variety admits two further contractions (corresponding to the divisors $\Hh$ and $\HH$) that are themselves Chow varieties formed with respect to the codimension two classes $\sigma_{1,1}$ and $\sigma_2$. The flip is a divisorial contraction of the Hilbert scheme contracting the divisor of nodal cubics by forgetting the embedded structure.
\end{remark}

\section{Degree three maps to Grassmannians of lines}\label{sec-three-line}

The discussion in the previous section does not apply to Grassmannians $G(k,n)$ when $k=2$.
In this section, we describe the necessary modifications for understanding the stable base locus decomposition for $\Kgnb{0,0}(G(2,n), 3)$, where $n \geq 5$. The class $\dunb$ need to be modified and the effective cone is no longer symmetric under interchanging $\sigma_{1,1}$ and $\sigma_2$.  The description and base loci of the divisor classes $S, P$ and $F$ remain unchanged. The divisors  described in \S 2 have the following expressions (see \cite{coskun:grass}).

$$T = \frac{2}{3}(\Hh + \HH + \Delta).$$

$$D_{deg} = \frac{1}{3}( - \Hh + 2 \HH - \Delta).$$

$$D_{unb} = \frac{1}{3} (5 \Hh - \HH - \Delta).$$
As in the previous section, the divisor class $P$ is defined as the pull-back of $\OO_{G(4,N)}(1)$ under the rational map that sends $f \in \Kgnb{0,0}(G(2,n),3)$ to the span of the image of $f$ in the Pl\"{u}cker embedding of $G(2,n)$. Intersecting $P$ with the test families obtained by taking  a pencil of conics in $\Sigma_{(1,1)^*}$ (or $\Sigma_{(2)^*}$) union a fixed line containing a base point and a pencil of lines in $\Sigma_{(1,1)^*}$ union a fixed conic attached at a base point, we see that
$$P = \frac{1}{3}(2\Hh + 2\HH - \Delta).$$
As long as the image of $f$ spans a three-dimensional projective space in the Pl\"{u}cker embedding of $G(2,n)$, $f$ is not contained in the base locus of $P$. Conversely, the argument given in the previous section shows that if the image of $f$ spans a linear space of dimension less than three, then $f$ is in the stable base locus of $P$. We conclude  that the stable base locus of $P$ is $\ccc[(1,1)^*] \cup \ccc[(2)^*] \cup \qqq((1)^*) \lll$.

Define the divisor class $F$ as the pull-back of the corresponding divisor in $\Kgnb{0,0}(\PP^N,3)$ introduced in \cite{dawei:cubic}. Then, by the argument given in Lemma \ref{disc-F}, $$F=  \frac{1}{3}(5\Hh + 5\HH  -\Delta)$$ and the stable base locus of $F$ is $\ccc[(1)^*] \cup \qqq((1)^*)\lll$. Define the divisor $S$ as in the previous section. The arguments  in Lemma \ref{disc-S} show that the class $S$ is given by
$$ S = \frac{1}{3}(-\Hh + 5 \HH - \Delta)$$
and the stable base locus of $S$ is equal to $\ccc[(1,1)^*] \cup \qqq((1,1)^*)\lll$. Finally, observe that the stable base locus of $\dunb$ is $\ccc[(3)^*] \cup \qqq((2)^*) \lll$ and the stable base locus of $\ddeg$ is $\ccc((2,1)^*) \cup \qqq((1,1)^*)\lll$.  First, suppose the domain of the stable map  $f$ is irreducible. As long as the pull-back of the tautological bundle of $G(2,n)$ has splitting type $(1,2)$, then $f$ is not in the indeterminacy locus of the map  $\phi$ defined in Section \ref{prelim}. Similarly, if the domain of $f$ has two components and the pull-back of the tautological bundle to the component of degree two has splitting type $(1,1)$, then $f$ is not in the indeterminacy locus of $\phi$. It follows that in both cases $f$ is not in the base locus of $\dunb$. If the domain of $f$ has three or four components, then the image could either consist of three concurrent lines or three non-concurrent lines where one line intersects the other two. It is easy to see that if the common point of intersection of the lines parameterize by two of the lines coincide, then $f$ is contained in $\qqq((2)^*)\lll$ and otherwise $f$ is not in the base locus of $\dunb$. An argument similar to the one in Lemma \ref{deg-unb} shows that $\ccc[(3)^*]\cup \qqq((2)^*)\lll$ is in the stable base locus of $\dunb$. The claim follows. The discussion of $\ddeg$ is similar.
\smallskip

\begin{notation}
Set $$ U = 2\Hh + 5\HH - \Delta, \ \ \ \mbox{and} \ \ \ U' = 5 \Hh + 2\HH - \Delta.$$
\end{notation}

The stable base locus decomposition of the effective cone of $\Kgnb{0,0}(G(2,n),3)$ has 15 chambers which are described in the following theorem. Figure \ref{fig:cublindecomp} depicts a cross-section of the effective cone.

\begin{figure}[ht!] 
   \centering
   \includegraphics[width=5in]{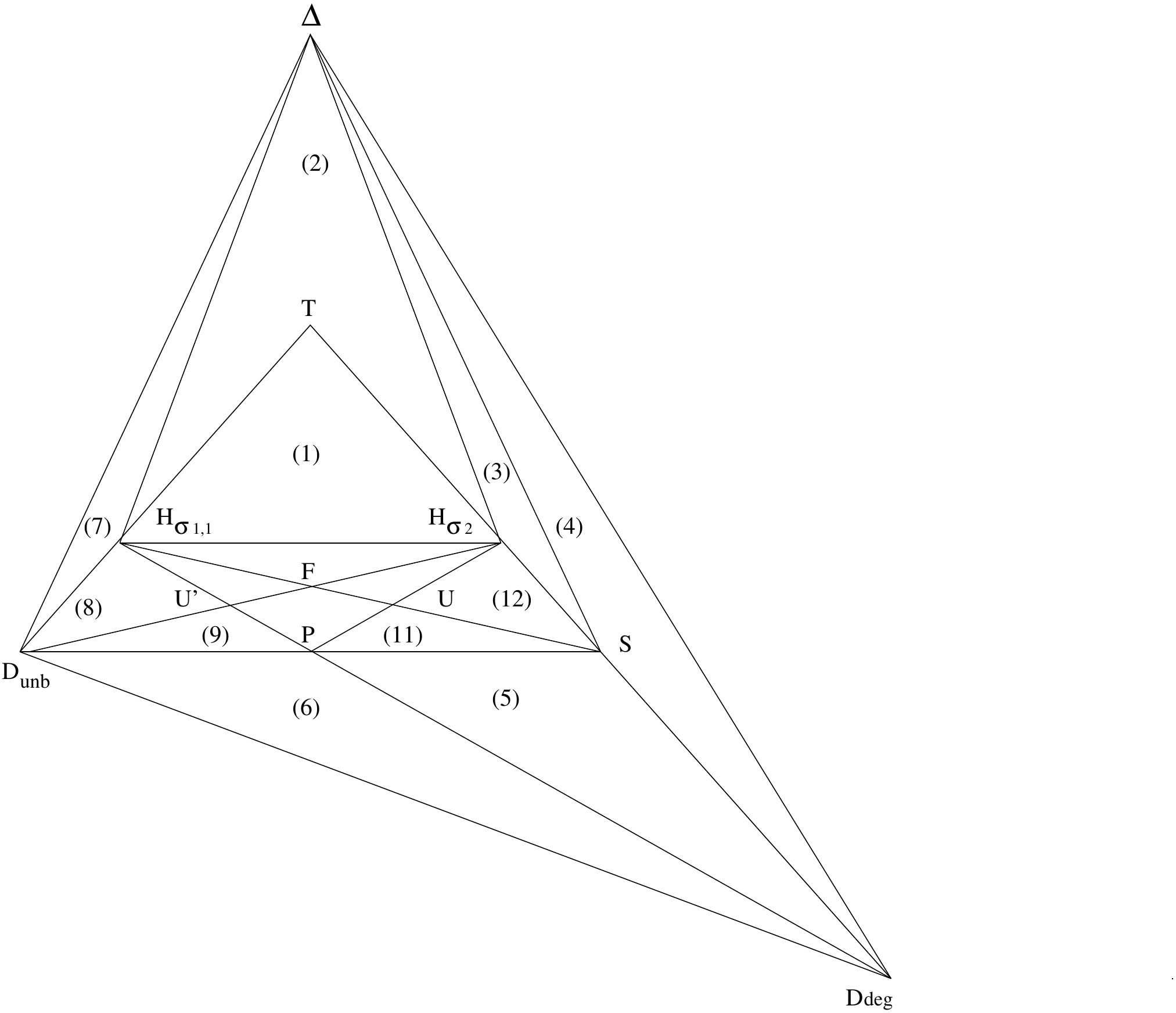}
   \caption{The stable base locus decomposition of $\Kgnb{0,0}(G(2,n),3)$.}
   \label{fig:cublindecomp}
\end{figure}

\begin{theorem}\label{de-three-line}
The stable base locus decomposition of the Neron-Severi space of $\Kgnb{0,0}(G(2,n), 3)$, with $n\geq 5$,  is given as follows:
\begin{enumerate}
\item In the closed cone spanned by  $\Hh, \HH $ and $T$,  the stable base locus is empty.
\smallskip

\item In the domain bounded by $\Delta, \Hh, T$ and $\HH$ union $c(\Hh \overline{\Delta}) \cup c(\HH \overline{\Delta})$, the stable base locus is equal to the boundary divisor.
\smallskip

\item In the domain bounded by $\HH, \Delta$ and $S$ union $c(\Delta S)$,  the stable base locus is  the union of $\ccc[(1,1)^*]$  and the boundary divisor.
\smallskip

\item In the domain bounded by $\ddeg, S$ and $\Delta$ union $c(\Delta D_{deg})$, the stable base locus is the union of  $\ccc[(2,1)^*]$ and the boundary divisor.
\smallskip

\item In the domain bounded by $D_{deg}, P$ and $S$ union $c(P \overline{D}_{deg}) \cup c(S \overline{D}_{deg})$,  the stable base locus consists of $\ccc[(2,1)^*] \cup \qqq((1,1)^*)\lll$.
\smallskip

\item In the domain bounded by $D_{unb}, D_{deg}$ and $P$ union $c(D_{deg} D_{unb})$, the stable base locus consists of the union $\ccc[(3)^*] \cup \qqq((2)^*)\lll \cup \ccc[(2,1)^*] \cup \qqq((1,1)^*)\lll$.
\smallskip

\item In the domain bounded by $D_{unb}, \Hh$ and $\Delta$ union $c(D_{unb}\Delta)$,  the stable base locus consists of  $\ccc[(3)^*]$ and the boundary divisor.
\smallskip

\item In the domain bounded by $D_{unb}, \Hh$ and $U'$ union $c(\Hh \overline{D}_{unb}) \cup c(U' \overline{D}_{unb})$, the stable base locus is $\ccc[(3)^*] \cup \qqq((2)^*)\lll$.
\smallskip

\item In the domain bounded by $D_{unb}, P$ and $U'$ union $c(P \dunb)$, the stable base locus is $\ccc[(3)^*] \cup \qqq((2)^*)\lll \cup \ccc[(1,1)^*]$
\smallskip

\item In the domain bounded by $P, U, F$ and $U'$ union $c(U\overline{P}) \cup c(U'\overline{P})$, the stable base locus consists of $\ccc[(1,1)^*]\cup\ccc[(2)^*]\cup \qqq((1)^*)\lll$.
\smallskip

\item In the domain bounded by $S, P$ and $U$ union $c(PS)$, the stable base locus consists of $\ccc[(1,1)^*]\cup \ccc[(2)^*] \cup \qqq((1,1)^*)\lll$.
\smallskip

\item In the domain bounded by $S, U$ and $\HH$  union $c(\HH \overline{S}) \cup c(U' \overline{S})$,  the stable base locus consists of $\ccc[(1,1)^*)  \cup \qqq((1,1)^*)\lll$.
\smallskip

\item In the domain bounded by $U,F$ and $\HH$ union $c(U\HH) \cup c(UF)$, the stable base locus consists of $\ccc[(1,1)^*) \cup \qqq((1)^*)\lll$.
\smallskip

\item In the domain bounded by $\Hh, F$ and $U'$ union $c(U'F) \cup c(U'\Hh)$, the stable base locus consists of $\ccc[(2)^*] \cup \qqq((1)^*)\lll$.
\smallskip

\item In the domain bounded by $\Hh, \HH$ and $F$ union $c(\Hh \overline{F}) \cup c(\HH \overline{F})$, the stable base locus consists of $\ccc[(1)^*] \cup \qqq((1)^*)\lll$.

\end{enumerate}
\end{theorem}

\begin{proof}
The proof of this theorem is very similar but easier than the proof of Theorem \ref{de-three}. Hence, we briefly sketch it and leave most of the details to the reader. The divisors $\Hh, \HH$ and $T$ are base-point-free, therefore, in the closed cone generated by these divisors the stable base locus is empty. Let $D= a\Hh + b \HH + c \Delta$. The curve classes $B_3$ and $B_{12}$ from the proof of Theorem \ref{de-three} show that the boundary divisor is in the stable base locus of any effective divisor contained in the complement of the closed cone generated by $\dunb, \ddeg$ and $T$. Since $\Hh$ and $\HH$ are base-point-free, in the domain bounded by $\Hh, T, \HH$ and $\Delta$ union $c(\Hh \overline{\Delta}) \cup c(\HH \overline{\Delta})$ the stable base locus consists of the boundary divisor. The curve induced by taking  the image of a general pencil of lines in the projection of the third Veronese embedding of $\PP^2$ to a plane Poincar\'{e} dual  to the class $\sigma_{(1,1)^*}$ shows that $\ccc[(1,1)^*]$ is in the stable base locus of $D$ if $a<0$. Hence, in the domain bounded by $S, \Delta$ and $\HH$ union $c(\Delta S)$, the stable base locus is the union of $\ccc[(1,1)^*]$ and the boundary divisor. Similarly, the curve induced by taking the image of a pencil of lines in the projection of the third Veronese embedding of $\PP^2$ to a $\PP^3$ Poincar\'{e} dual to the class $\sigma_{(3)^*}$ shows that $\ccc[(3)^*]$ is in the stable base locus of $D$ if $b<0$. We conclude that in the region bounded by $\dunb, \Delta$ and $\Hh$ union $c(\dunb \Delta)$, the stable base locus is the union of $\ccc[(3)^*]$ and the boundary divisor.
Let $A_1$ be the curve class induced in $\Kgnb{0,0}(G(2,n),3)$ by a pencil of cubic surfaces in $\PP^3$ with a fixed double line and 8 general base points. Then  $$A_{1} \cdot \Hh = 5, \ \ A_{1} \cdot \HH = 1, \ \ A_{1} \cdot \Delta = 0.$$ The last two of these equalities are clear. The first one may be computed using the identity $A_{1} \cdot S = 0$. Curves in the class $A_{1}$ cover $\ccc[(2,1)^*]$, so $\ccc[(2,1)^*]$ is in the stable base locus of $D$ if $5a + b <0$. Therefore, in the domain bounded by $\ddeg, S$ and $\Delta$ union $c(\ddeg \Delta)$, the stable base locus is $\ccc[(2,1)^*]$ union the boundary divisor.
\smallskip

The curve classes $B_6$ and $B_7$ introduced during the proof of Lemma \ref{disc-F} show that $\ccc[(1)^*]$ and $\qqq((1)^*)\lll$ are in the stable base locus of $D$ if $c<0$. Since the stable base locus of $F$ is equal to the union of these two loci,  in the region bounded by $F, \Hh$ and $\HH$ union $c(\Hh\overline{F}) \cup c(\HH \overline{F})$, the stable base locus equals $\ccc[(1)^*] \cup \qqq((1)^*)\lll$. Let $A_2$ be the curve class induced in $\Kgnb{0,0}(G(2,n),3)$ by taking the cone over a pencil of twisted cubic curves in a fixed quadric surface in $\PP^3$.   Since $$A_2 \cdot \Hh = 0, \ \ A_2 \cdot \HH = 2, \ \ A_2 \cdot \Delta = 4,$$ $\ccc[(3)^*]$ is in the stable base locus of $D$ if $b<-2c$. The curve class $B_{15}$ introduced in the proof of Theorem \ref{de-three} shows that $\ccc[(2,1)^*]$ is in the stable base locus of $D$ if $a+b+4c<0$. Similarly, the loci $\qqq((1,1)^*)\lll$ and $\qqq((2)^*)\lll$ are in the base locus of $D$ if $a<-2c$ and $b<-2c$, respectively. We conclude that in the domain bounded by $\dunb, U'$ and $\Hh$ union $c(U' \overline{D}_{unb}) \cup c(\Hh \overline{D}_{unb})$, the stable base locus is $\ccc[(3)^*] \cup \qqq((2)^*)\lll$. In the domain bounded by $\dunb, P$ and $\ddeg$ union $c(\dunb \ddeg)$, the stable base locus is $\ccc[(3)^*] \cup \ccc[(2,1)^*] \cup \qqq((2)^*)\lll \cup \qqq((1,1)^*)\lll$. In the domain bounded by $\ddeg, P$ and $S$ union $c(P\overline{D}_{deg}) \cup c(S \overline{D}_{deg})$, the stable base locus is $\ccc[(2,1)^*] \cup \qqq((1,1)^*)\lll$.
\smallskip

The curve classes $B_4$ and $B_5$ introduced in the proof of Lemma \ref{disc-P} show that $\ccc[(1,1)^*]$ and $\ccc[(2)^*]$, respectively, are in the stable base locus of $D$ if $a+5c<0$ and $b+5c<0$, respectively. Hence, in the domain bounded by $P, U', F$ and $P$ union $c(U' \overline{P}) \cup c(U \overline{P})$, the stable base locus is $\ccc[(1,1)^*]\cup \ccc[(2)^*] \cup \qqq((1)^*)\lll$. The stable base locus of a divisor contained in the domain bounded by $\dunb, P$ and $U'$ union $c(\dunb P)$ is a subset of the union of the stable base loci of $\dunb$ and $P$. Therefore, in this region the stable base locus is $\ccc[(3)^*] \cup \ccc[(1,1)^*] \cup \qqq((2)^*)\lll$. Similarly, in the domain bounded by $P,S$ and $U$ union $c(PS)$, the stable base locus is $\ccc[(2)^*] \cup \ccc[(1,1)^*] \cup \qqq((1,1)^*)\lll$.
\smallskip

The stable base locus of $U$ (resp., $U'$) is contained in the intersection of the stable base loci of $S$ and $P$ (resp., $\dunb$ and $P$). Moreover, in the domain bounded by $U, \HH$ and $F$ union $c(F\overline{U}) \cup c(\HH \overline{U})$  (resp., $U', \Hh$ and $F$ union $c(F\overline{U}') \cup c(\Hh\overline{U}')$) the stable base locus is contained in the stable base locus of $U$ (resp., $U'$).
It follows that the stable base loci are $\ccc[(1,1)^*]\cup \qqq((1)^*)\lll$ and $\ccc[(2)^*] \cup \qqq((1)^*)\lll$, respectively. Finally, in the domain bounded by $S, U$ and $\HH$ union $c(U\overline{S}) \cup c(\HH \overline{S})$, the stable base locus is contained in that of $S$. Hence, in this region the stable base locus is $\ccc[(1,1)^*] \cup \qqq((1,1)^*)\lll$.  This concludes the proof of the theorem.
\smallskip

\end{proof}

\begin{remark}
The description of the models is analogous to the case of $\Kgnb{0,0}(G(k,n),3)$ with $k\geq 3$ described in Remark \ref{remark-cube}. We leave the necessary modifications to the reader.
\end{remark}

\bigskip

\end{document}